\documentclass[12pt]{amsart}

\usepackage[utf8x]{inputenc}
\usepackage[english]{babel}
\usepackage{amsmath, amsfonts, amssymb}
\usepackage{color}
\usepackage{tikz}
\usetikzlibrary{positioning,shapes,shadows,arrows}
\usepackage{enumitem}
\usepackage{hhline}
\usepackage{array}
\usepackage[colorlinks=true,linkcolor=blue,citecolor=blue,urlcolor=blue]{hyperref}
\usepackage{dpfloat}
\usepackage{fancyhdr}
\usepackage{algpseudocode}
\usepackage{cite}
\usepackage{listings}
\usepackage{appendix}
\usepackage{shuffle}

\title[Counting smaller elements in Tamari lattices]{Counting smaller elements in the Tamari and $m$-Tamari lattices}

\author{Gr\'egory Chatel, Viviane Pons}

\address{Laboratoire d'Informatique Gaspard Monge, Université Paris-Est
    Marne-la-Vallée. \\ Fakultät für Mathematik, Universität Wien}
    
\keywords{binary trees, Tamari lattice, Tamari intervals}

\numberwithin{equation}{section}
\newtheorem{Theoreme}{Theorem}[section]

\newtheorem{Proposition}[Theoreme]{Proposition}
\newtheorem{Lemme}[Theoreme]{Lemma}
\newtheorem{Definition}[Theoreme]{Definition}

\newcommand{\NN}{\mathbb{N}}

\DeclareMathOperator{\PBT}{\textbf{PBT}}
\DeclareMathOperator{\HT}{\textbf{H}}
\DeclareMathOperator{\ET}{\textbf{E}}
\DeclareMathOperator{\PT}{\textbf{P}}
\DeclareMathOperator{\B}{B}
\DeclareMathOperator{\BB}{\mathbb{B}}
\DeclareMathOperator{\BT}{\mathcal{B}}
\DeclareMathOperator{\Bm}{B^{(m)}}

\newcommand{\Bk}[1]{\B^{(#1)}} 
\DeclareMathOperator{\BR}{R} 
\DeclareMathOperator{\BBm}{\mathbb{B}^{(m)}}
\newcommand{\BTm}[1]{\BT_{#1}^{(m)}} 
\newcommand{\BTk}[2]{\BT_{#1}^{(#2)}}
\newcommand{\Phim}{\Phi^{(m)}} 

\DeclareMathOperator{\PI}{\mathcal{P}}
\DeclareMathOperator{\PIm}{\mathcal{P}^{(m)}}
\newcommand{\Tamnm}{\mathcal{T}_{n}^{(m)}} 
\newcommand{\Tam}[2]{\mathcal{T}_{#1}^{(#2)}}
\newcommand{\ExtL}{\mathrm{ExtL}}

\newcommand{\trprec}{\vartriangleleft} 
\newcommand{\ntrprec}{\ntriangleleft}

\DeclareMathOperator{\Itrees}{trees} 
\DeclareMathOperator{\Isize}{size} 
\newcommand{\pleft}{\vec{\bullet}} 
\newcommand{\pright}{\overleftarrow{\delta}} 
\newcommand{\polleft}{\succ}
\newcommand{\polright}{\prec_\delta}
\newcommand{\polrightx}{\prec_{\frac{\delta}{x}}} 
\newcommand{\prightx}{\tfrac{\pright}{x}}
\newcommand{\SI}{\mathbb{S}} 

\newcommand\dec{F_{\ge}}
\newcommand\inc{F_{\le}}

\newcolumntype{C}{>{\centering\arraybackslash} m{3.2cm}}
\newcolumntype{D}{>{\centering\arraybackslash} m{3.5cm}}

\definecolor{darkGreen}{RGB}{23,103,1}

\newcommand{\red}[1]{\textbf{\textcolor{red}{#1}}}

\newcommand{\blue}[1]{\textcolor{blue}{#1}}
\newcommand{\green}[1]{\textcolor{darkGreen}{#1}}

\tikzstyle{Red} = [color = red]
\tikzstyle{Blue} = [color = blue]
\tikzstyle{Green} = [color = darkGreen]
\tikzstyle{Gray} = [color = gray]

\tikzstyle{Path} = [line width = 1.2]
\tikzstyle{StrongPath} =  [line width=2]
\tikzstyle{DPoint} = [fill, radius=0.1]
\tikzstyle{Line1} = [dashed]
\tikzstyle{Line2} = [dotted, ultra thick]

\tikzstyle{Point} = [fill, radius=0.08]
\tikzstyle{RedPoint} = [color = red, fill, radius=0.08]
\tikzstyle{BluePoint} = [color = blue, fill, radius=0.08]
\tikzstyle{GreenPoint} = [color = darkGreen, fill, radius=0.08]
\tikzstyle{RedPath} = [color = red]
\tikzstyle{BluePath} = [color = blue]
\tikzstyle{GreenPath} = [color = darkGreen]
\tikzstyle{GrayPath} = [color = gray]
\tikzstyle{StrongPath} =  [line width=2.5]
\tikzstyle{StrongPath2} =  [line width=1.5]
\tikzstyle{Leaf} = [color = gray]
\tikzstyle{RedLabel} = [color = red]
\tikzstyle{BlueLabel} = [color = blue]
\tikzstyle{GreenLabel} = [color = darkGreen]
\tikzstyle{Label} = [color = black]
\begin{document}

\maketitle

\begin{abstract}
We introduce new combinatorial objects, the interval-posets, that encode intervals of the Tamari lattice. We then find a combinatorial interpretation of the bilinear operator that appears in the functional equation of Tamari intervals described by Chapoton. Thus, we retrieve this functional equation and prove that the polynomial recursively computed from the bilinear operator on each tree $T$ counts the number of trees smaller than $T$ in the Tamari order. 

Then we show that a similar $(m+1)$-linear operator is also used in the functional equation of $m$-Tamari intervals. We explain how the $m$-Tamari lattices can be interpreted in terms of $(m+1)$-ary trees or a certain class of binary trees. We then use the interval-posets to recover the functional equation of $m$-Tamari intervals and to prove a generalized formula that counts the number of elements smaller than or equal to a given tree in the $m$-Tamari lattice.  
\end{abstract}

\section{Introduction}
\label{sec:Intro}

The combinatorics of planar binary trees is known to have very interesting algebraic properties. Loday and Ronco first introduced the Hopf Algebra $\PBT$ based on these objects \cite{PBT1}. It was re-built by Hivert, Novelli and Thibon \cite{PBT2} through the introduction of the sylvester monoid. The structure of $\PBT$ involves a very nice object which is connected to both algebra and classical algorithmic: the \emph{Tamari lattice}. 

It was introduced by Tamari himself in 1962 as an order on formal bracketings \cite{Tamari1} and was proved later to be a lattice \cite{Tamari2}. It can be realized as a polytope called the associahedron. On binary trees, it can be described by a very common operation in algorithmic: the \emph{right rotation} (see Figure \ref{fig:tree-right-rotation}). More generally, the cover relations of the Tamari order can be translated to many other combinatorial objects counted by Catalan numbers \cite{Stanley}, like Dyck paths.

In this paper, we study the enumeration of the intervals of the Tamari lattice. Surprisingly, the number of intervals is given by a very beautiful formula
\begin{equation}
\label{eq:intervals-formula}
I_n =  \frac{2}{n(n +1)} \binom{4 n + 1}{n - 1},
\end{equation}
where $I_n$ is the number of intervals of the Tamari lattice of binary trees of size $n$. It was proven by Chapoton \cite{Chap} using a functional equation on the generating function of the intervals. Very recently, Bergeron and Préville-Ratelle introduced a new set of lattices generalizing the Tamari lattice \cite{BergmTamari}. They are called the $m$-Tamari lattices and their elements are counted by the $m$-Catalan numbers. In this case also, one can obtain a formula counting the number intervals
\begin{equation}
\label{eq:m-intervals-formula}
I_{n,m} = \frac{m+1}{n(mn +1)} \binom{(m+1)^2 n + m}{n - 1}.
\end{equation} 
This was conjectured in \cite{BergmTamari} and proved in \cite{mTamari}. The proof also uses a functional equation that generalizes the classical case studied by Chapoton.

Here, we propose refined versions of both results by studying a new object that we call \emph{interval-poset}. Each interval-poset corresponds to an interval of the Tamari lattice. To construct these objects, we use the strong relations between the Tamari order and the weak order on permutations. It has been known since Bj\"orner and Wachs \cite{BW} that linear extensions of a certain labelling of binary trees correspond to intervals of the weak order on permutations. This was more explicitly described in \cite{PBT2} with sylvester classes. The elements of the basis $\PT$ of $\PBT$ are indexed by binary trees and defined as a sum on a sylvester class of elements of $\textbf{FQSym}$. The $\PBT$ algebra also admits two other bases $\HT$ and $\ET$. An element of $\HT$ (resp. $\ET$) is a sum of elements $\PT_T$ over an initial (resp. final) interval of Tamari lattice. They can be indexed by planar forests and, with a well chosen labelling, their linear extensions are intervals of the weak order on permutations corresponding to a reunion of sylvester classes. By combining the forests of the initials and finals intervals of two comparable trees in one single poset, we obtain what we call an interval-poset. Its linear extensions are exactly the sylvester classes corresponding to the interval in the weak order. This new object has nice combinatorial properties and allows to perform computations on Tamari intervals.

Thereby, we give a new proof of the formula of Chapoton \eqref{eq:intervals-formula}. This proof is based on the study of a bilinear operator that already appeared in \cite{Chap} but was not explored yet. It leads to the definition of a new family of polynomials:

\begin{Definition}
\label{def:tamari-polynomials}
Let $T$ be a binary tree, the polynomial $\BT_T(x)$ is recursively defined by
\begin{align}
\notag
\BT_\emptyset &:= 1 \\
\notag
\BT_T(x) &:= x \BT_L(x) \frac{x \BT_R(x) - \BT_R(1)}{x - 1}
\end{align}
where $L$ and $R$ are respectively the left and right subtrees of $T$. We call $\BT_T(x)$ the \emph{Tamari polynomial} of $T$ and the set of \emph{Tamari polynomials} is the image of the map $T \mapsto \BT_T(x)$.
\end{Definition}

This family of polynomials is yet unexplored in this context but seems to appear in a different computation made by Chapoton on rooted trees \cite{ChapBiVar}. We give all polynomials for binary trees of size $n\leq4$ in Figure \ref{fig:tam-poly}. Our approach on Tamari interval-posets allows us to prove the following theorem in Section \ref{sec:main-result}:

\begin{Theoreme}
\label{thm:smaller-trees}
Let $T$ be a binary tree. Its Tamari polynomial $\BT_T(x)$ counts the trees smaller than or equal to $T$ in the Tamari order according to the number of nodes on their leftmost branch. In particular, $\BT_T(1)$ is the number of trees smaller than or equal to $T$.

Symmetrically, if $\tilde{\BT}_T$ is defined by exchanging the role of left and right children in Definition \ref{def:tamari-polynomials}, then it counts the number of trees greater than or equal to $T$ according to the number of nodes on their right border.
\end{Theoreme}

It was shown in \cite{mTamari} that the $m$-Tamari lattices can be seen as ideals of the Tamari lattice of size $n \times m$. Therefore, an interval of the $m$-Tamari lattice is an interval of Tamari which satisfies some conditions. This can be expressed in terms of interval-posets. Thus, it allows us to easily generalize our results to the $m$-Tamari case. We re-obtain the functional equation on the generating function described in \cite{mTamari} along with a generalization of Theorem \ref{thm:smaller-trees} to count smaller elements in the $m$-Tamari lattices.

We first recall in Section \ref{sec:def} some definitions and properties of the Tamari lattice. We then introduce the notion of \emph{interval-poset} to encode a Tamari interval. In Section \ref{sec:tamari-polynomials}, we show the implicit bilinear operator that appears in the functional equation  of the generating functions of Tamari intervals. We then explain how interval-posets can be used to give a combinatorial interpretation of this bilinear operator and thereby give a new proof of the functional equation. Theorem \ref{thm:smaller-trees} follows naturally. In Section \ref{sec:bivar}, we discuss the similarity between Tamari polynomials and some bivariate polynomials that appeared in the context of flows of rooted trees \cite{ChapBiVar}.

Section \ref{sec:mTamari} is dedicated to the study of the $m$-Tamari lattices defined in \cite{BergmTamari}. A functional equation for the intervals of these lattices is shown in \cite{mTamari} and contains a $m\!+\!1$-linear operator that generalizes the binary case. In Section \ref{sec:m-trees}, we explain how the $m$-Tamari lattice can be seen on a certain class of binary trees which are in bijection with $(m+1)$-ary trees. Thus, we are able to use again the interval-posets with a generalized combinatorial $m\!+\!1$-operator to reobtain the functional equation of intervals of the $m$-Tamari order. We then prove Theorem \ref{thm:smaller-mtrees}, the generalization of Theorem \ref{thm:smaller-trees} for the $m$-Tamari order.

\begin{figure}[ht]

\def \sscale{0.2}
\def \fpath{figures/trees/}

\begin{tabular}{ccc}
{$\!
\begin{aligned}
\BT_{\scalebox{\sscale}{


\begin{tikzpicture}
\node (N0) at (1.250, 0.000){};
\node (N00) at (0.750, -0.500){};
\draw[Point] (N00) circle;
\draw (N0.center) -- (N00.center);
\draw[Point] (N0) circle;
\end{tikzpicture}}}(x) &= x^2 \\
\BT_{\scalebox{\sscale}{


\begin{tikzpicture}
\node (N0) at (0.250, 0.000){};
\node (N01) at (0.750, -0.500){};
\draw[Point] (N01) circle;
\draw (N0.center) -- (N01.center);
\draw[Point] (N0) circle;
\end{tikzpicture}}}(x) &= x^2 + x \\
\BT_{\scalebox{\sscale}{


\begin{tikzpicture}
\node (N0) at (1.250, 0.000){};
\node (N00) at (0.750, -0.500){};
\node (N000) at (0.250, -1.000){};
\draw[Point] (N000) circle;
\draw (N00.center) -- (N000.center);
\draw[Point] (N00) circle;
\draw (N0.center) -- (N00.center);
\draw[Point] (N0) circle;
\end{tikzpicture}}}(x) &= x^3\\
\BT_{\scalebox{\sscale}{


\begin{tikzpicture}
\node (N0) at (1.250, 0.000){};
\node (N00) at (0.250, -0.500){};
\node (N001) at (0.750, -1.000){};
\draw[Point] (N001) circle;
\draw (N00.center) -- (N001.center);
\draw[Point] (N00) circle;
\draw (N0.center) -- (N00.center);
\draw[Point] (N0) circle;
\end{tikzpicture}}}(x) &= x^3 + x^2\\
\BT_{\scalebox{\sscale}{


\begin{tikzpicture}
\node (N0) at (0.250, 0.000){};
\node (N01) at (1.250, -0.500){};
\node (N010) at (0.750, -1.000){};
\draw[Point] (N010) circle;
\draw (N01.center) -- (N010.center);
\draw[Point] (N01) circle;
\draw (N0.center) -- (N01.center);
\draw[Point] (N0) circle;
\end{tikzpicture}}}(x) &= x^3 +x^2 + x\\
\BT_{\scalebox{\sscale}{

%

\begin{tikzpicture}
\node (N0) at (0.750, 0.000){};
\node (N00) at (0.250, -0.500){};
\draw[Point] (N00) circle;
\node (N01) at (1.250, -0.500){};
\draw[Point] (N01) circle;
\draw (N0.center) -- (N00.center);
\draw (N0.center) -- (N01.center);
\draw[Point] (N0) circle;
\end{tikzpicture}}}(x) &= x^3 + x^2\\
\BT_{\scalebox{\sscale}{


\begin{tikzpicture}
\node (N0) at (0.250, 0.000){};
\node (N01) at (0.750, -0.500){};
\node (N011) at (1.250, -1.000){};
\draw[Point] (N011) circle;
\draw (N01.center) -- (N011.center);
\draw[Point] (N01) circle;
\draw (N0.center) -- (N01.center);
\draw[Point] (N0) circle;
\end{tikzpicture}}}(x) &= x^3 +2x^2 + 2x\\
\end{aligned}$}
&
{$\!
\begin{aligned}
\BT_{\scalebox{\sscale}{\input{\fpath T4-1}}}(x)  &= x^4 \\
\BT_{\scalebox{\sscale}{\input{\fpath T4-2}}}(x)  &= x^4 + x^3 \\
\BT_{\scalebox{\sscale}{\input{\fpath T4-3}}}(x)  &= x^4 + x^3 \\
\BT_{\scalebox{\sscale}{\input{\fpath T4-4}}}(x)  &= x^4 + x^3 + x^2 \\
\BT_{\scalebox{\sscale}{\input{\fpath T4-5}}}(x)  &= x^4 + 2x^3 + 2x^2 \\
\BT_{\scalebox{\sscale}{\input{\fpath T4-6}}}(x)  &= x^4 + x^3 \\
\BT_{\scalebox{\sscale}{\input{\fpath T4-7}}}(x)  &= x^4 + 2x^3 + x^2 \\
\end{aligned}$}
&
{$\!
\begin{aligned}
\BT_{\scalebox{\sscale}{\input{\fpath T4-8}}}(x)  &= x^4 + x^3 + x^2 \\
\BT_{\scalebox{\sscale}{\input{\fpath T4-9}}}(x)  &= x^4 + 2x^3 + 2x^2 \\
\BT_{\scalebox{\sscale}{\input{\fpath T4-10}}}(x)  &= x^4 + x^3 + x^2 + x \\
\BT_{\scalebox{\sscale}{\input{\fpath T4-11}}}(x)  &= x^4 + 2x^3 + 2x^2 + 2x \\
\BT_{\scalebox{\sscale}{\input{\fpath T4-12}}}(x)  &= x^4 + 2x^3 + 2x^2 + 2x\\
\BT_{\scalebox{\sscale}{\input{\fpath T4-13}}}(x)  &= x^4 + 2x^3 + 3x^2 + 3x \\
\BT_{\scalebox{\sscale}{\input{\fpath T4-14}}}(x)  &= x^4 + 3x^3 + 5x^2 + 5x\\
\end{aligned}$}
\end{tabular}

\caption{Tamari polynomials for binary trees of size~$n\leq4$.}
\label{fig:tam-poly}
\end{figure}

\section{Interval-posets of Tamari lattice}
\label{sec:def}

\subsection{The Tamari order on paths and binary trees}
\label{sub-sec:tamari}

Originally, the Tamari lattice has been described on bracketing \cite{Tamari1} but it is also commonly defined on Dyck paths.

\begin{Definition}
A \emph{Dyck path} of size $n$ is a lattice path from the origin $(0,0)$ to the point $(2n,0)$ made from a sequence of up steps $(1,1)$ and down steps $(1,-1)$ such that the path stays above the line $y=0$. 
\end{Definition}

A Dyck path can also be considered as a binary word by replacing up steps by the letter $1$ and down steps by $0$. We call a Dyck path \emph{primitive} if it only touches the line $y=0$ on its end points. A \emph{rotation} consists of switching a down step $d$ with the primitive Dyck path starting right after $d$, see Figure \ref{fig:rot-dyck}.

\begin{figure}[ht]
\scalebox{0.8}{
\input{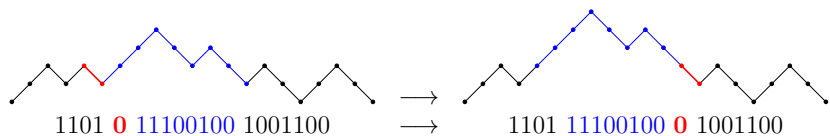}
}
\caption{Rotation on Dyck Paths.}
\label{fig:rot-dyck}
\end{figure}

The Tamari order on Dyck paths is defined as the transitive and reflexive closure of the rotation operation: a path $D'$ is greater than a path $D$ if it can be obtained by applying a sequence of right rotation on $D$. It is indeed an order and even a lattice \cite{Tamari1, Tamari2}. See Figure \ref{fig:tamari-dyck} for the lattices on Dyck paths of sizes 3 and 4.

\begin{figure}[ht]
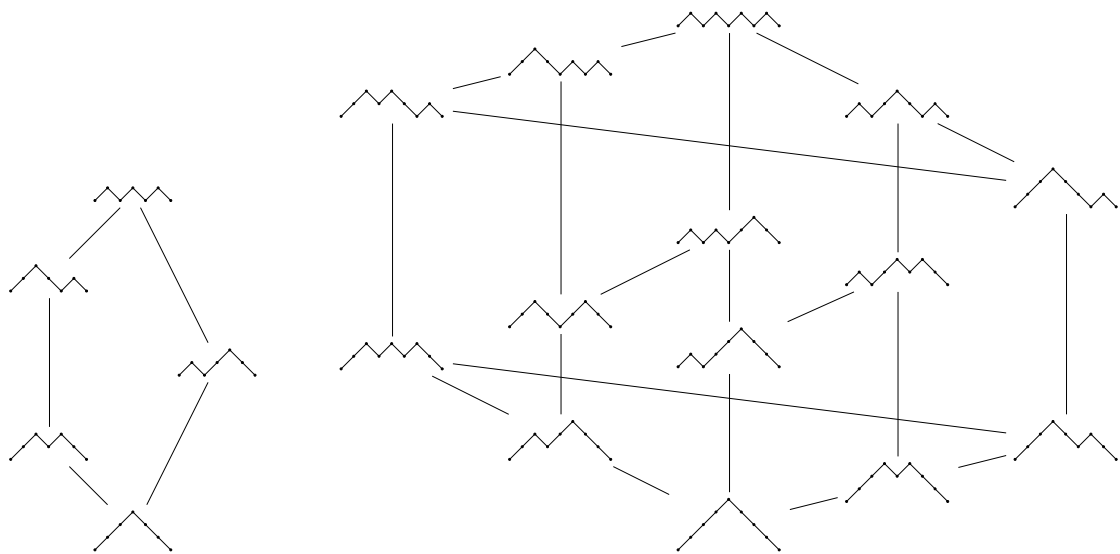

  
\hspace*{-2cm}
    \begin{tabular}{cc}
    \scalebox{0.7}{\input{figures/tamari_dyck-3}}&
    \scalebox{0.7}{\input{figures/tamari_dyck-4}}
    \end{tabular}
    
    \caption{Tamari lattices of sizes 3 and 4 on Dyck paths.}
    
    \label{fig:tamari-dyck}

\end{figure}

A binary tree is recursively defined by being either the empty tree~$(\emptyset)$ or a pair of  binary trees, respectively called \emph{left} and \emph{right} subtrees, grafted on an internal node\footnote{Note that what we call binary tree is actually a planar binary tree. All binary trees consider in this paper are planar, \emph{i.e.}, the subtrees are ordered.}. If a tree $T$ is composed of a root node $x$ with $A$ and $B$ as respectively left and right subtrees, we write $T = x(A,B)$. The number of nodes of a tree $T$ is called the size of $T$.

There are many ways to define a bijection between Dyck paths and binary trees. The one we use here is the only one which is consistent with the usual definition of the Tamari order on binary trees through the right rotation (see Definition \ref{def:tree-rotation} later). Similarly to a binary tree, a Dyck path can be seen as a recursive binary object: it is either an empty path or a word $D_1 1 D_2 0$ where $D_1$ and $D_2$ are two Dyck paths (potentially empty ones). The subpath $D_1$ corresponds to the left factor of $D$ up to the last touching point of $D$ before the end. Consequently, if $D$ is primitive, $D_1$ is empty. If both $D_1$ and $D_2$ are empty, then $D$ is the only dyck path of size 1: the word $10$. We define recursively the binary tree $T$ corresponding to $D$. If $D$ is the empty word, then $T$ is the empty tree. Otherwise, $T$ is a binary tree whose left subtree (resp. right rubstree) corresponds to $D_1$ (resp. $D_2$). See Figure \ref{fig:dyck-tree} for an example of the bijection. 

\begin{figure}[ht]
\centering
\scalebox{0.8}{
\input{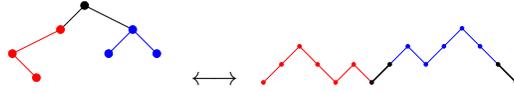}
}
    \caption{Bijection between Dyck paths and binary trees.}
    
    \label{fig:dyck-tree}
    
\end{figure}

Through this bijection, the rotation on Dyck paths can be interpreted directly in terms of binary trees through an operation called the \emph{right rotation}. This is a well known operation on binary trees, used in many different contexts, especially sorting algorithms \cite{AVL}.

\begin{Definition}
\label{def:tree-rotation}
Let $y$ be a node of $T$ with a non-empty left subtree $x$. The \emph{right rotation} of $T$ on $y$ is a local rewriting which follows Figure~\ref{fig:tree-right-rotation}, that is replacing $y( x(A,B), C)$ by $x(A,y(B,C))$ (note that $A$, $B$, or $C$ might be empty).

\begin{figure}[ht]
\centering

\input{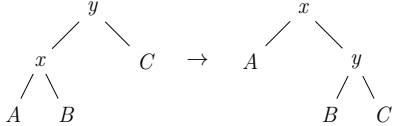}

\caption{Right rotation on a binary tree.}

\label{fig:tree-right-rotation}

\end{figure}

\end{Definition}

The right rotation is then the cover relation of the Tamari order on binary trees, as illustrated in Figure \ref{fig:tamari-trees}.

\begin{figure}[ht]
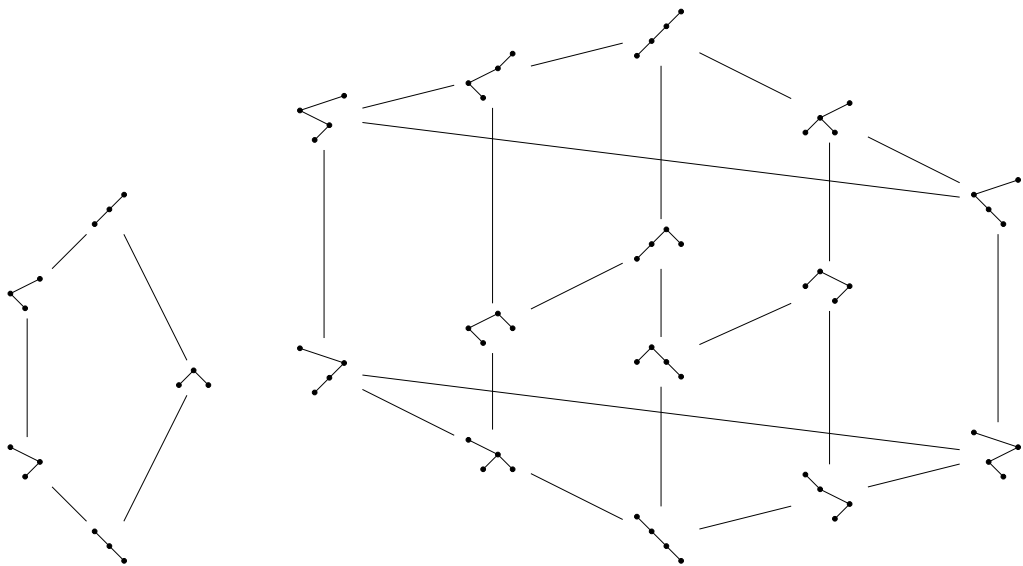

  
\hspace*{-1.5cm}
    \begin{tabular}{cc}
    \scalebox{0.7}{\input{figures/tamari_trees-3}}&
    \scalebox{0.7}{\input{figures/tamari_trees-4}}
    \end{tabular}
    
    \caption{Tamari lattice of sizes 3 and 4 on binary trees.}
    
    \label{fig:tamari-trees}

\end{figure}

\subsection{Relation with the weak order.}

An interesting property is the relation between the Tamari lattice and the weak order on permutations. Indeed, the Tamari lattice is a sublattice of the right weak order: it can be induced from it by choosing the proper subset of permutations. It is also a quotient lattice: one can define a relation on permutations given by a surjective map to binary trees. The quotient lattice of the right weak order by this relation gives the Tamari lattice. These results are originally due to Tonks \cite{Tonks}. They are also explained in \cite{PBT2}. In this paper, we use the combinatorial constructions of the latter that we recall now. They are based on a very classical object in computer science: \emph{binary search trees}.

\begin{Definition}
\label{def:binary-search-tree}
A \emph{binary search tree} is a labelled binary tree where for each node of
label $k$, any label in its left (resp. right) subtree is smaller
than or equal to (resp. larger than) $k$.
\end{Definition}

\begin{figure}[p]
\centering
\begin{fullpage}
\input{figures/tamari_quotient3}
\caption{The Tamari order of size 3 as ad 4 a quotient of the weak order.}
\label{fig:tamari-quotient}
\end{fullpage}
\end{figure}

\begin{figure}[p]
\centering
\begin{fullpage}
\scalebox{0.9}{\input{figures/tamari_quotient4}}
\end{fullpage}
\end{figure}

Note that, in general, binary search trees are labelled by any set of numbers, allowing repetitions. However, we will only consider binary search trees with distinct labels. Figure \ref{fig:bst-example} shows an example of such a tree. There is only one way to label a binary tree of size $n$ with distinct labels $1, \dots, n$ to make it a binary search tree. We call this the \emph{binary search tree labelling} of the tree and often identify the two objects. Such a labelled tree can be interpreted as a poset. The order relation, denoted  $\trprec$\footnote{We use the notation $\trprec$ for all posets of integers to differentiate with the natural order on integers. When necessary, we index the notation by the name of the  object. If $T$ is a tree (or a forest or a poset), $\trprec_T$ is the order relation given by the tree $T$.}, is defined by $x \trprec y$ if and only if $x$ is in the subtree whose root is $y$.  For example,
the tree
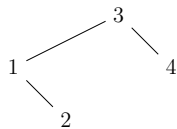
\begin{center}
\scalebox{0.7}{


\begin{tikzpicture}
\node (N0) at (2.500, 0.000){3};
\node (N00) at (0.500, -1.000){1};
\node (N001) at (1.500, -2.000){2};
\draw (N00) -- (N001);
\node (N01) at (3.500, -1.000){4};
\draw (N0) -- (N00);
\draw (N0) -- (N01);
\end{tikzpicture}
  }
\end{center}
is the poset where $2 \trprec 1 \trprec 3$ and $4\trprec 3$. A linear extension of this poset is a permutation of the labels of the tree
where for all labels $a$ and $b$, if $a \trprec b$ in the poset, then the number $a$ is before $b$
in the permutation. For example, $4213, 2413$, and $2143$ are the three linear extensions of the above tree. The permutation $1423$ is not because $1$ appears before $2$ whereas $2 \trprec 1$. The set of linear extensions of a given tree is called the \emph{sylvester class} of the tree: it forms an interval of the right weak order as illustrated in Figure~\ref{fig:bst-example}. Indeed, the set of linear extensions of a tree can be computed recursively using the concatenation and shuffle product which are stable operations on intervals of the weak order, for example, in Figure~\ref{fig:bst-example} the set of linear extensions is given by~$(21 \shuffle 4).3$. The sylvester classes of the binary trees of size $n$ form a partition of $\mathfrak{S}_n$. The ordering between classes is well defined as a quotient order of the weak order and it corresponds to the Tamari order on binary trees \cite{Tonks, PBT2}. This is illustrated by Figure \ref{fig:tamari-quotient}: two binary trees $T_1$ and $T_2$ are such that $T_1 \leq T_2$ if and only if there exists two linear extensions $\sigma_1$ and $\sigma_2$ of respectively $T_1$ and $T_2$ such that $\sigma_1 \leq \sigma_2$ for the right weak order.

\begin{figure}[ht]

  \begin{center}

    \input{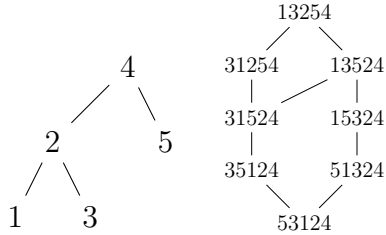}

    \caption{A binary search tree and its corresponding
      sylvester class.}
    \label{fig:bst-example}
    
  \end{center}

\end{figure}

\subsection{Construction of interval-posets}
\label{sec:interval-posets}

We now introduce more general objects: \emph{interval-posets} in bijection with the intervals of the Tamari order. Let us first recall two bijections between binary search trees and forests of planar trees. We say that a binary search tree has an \emph{increasing} relation between $a$ and $b$ if $a<b$ and $a \trprec b$, which means $a$ is in the left subtree of $b$. Symmetrically, a binary search tree has a \emph{decreasing} relation if $a<b$ and $b \trprec a$, \emph{i.e.}, $b$ is in the right subtree of $a$. From a binary search tree $T$, one can construct a poset containing only increasing (resp. decreasing) relations of $T$. These posets are actually forests, we call them the \emph{initial} and \emph{final} forest of the binary tree.
\begin{Definition}
\label{def:LSRB-RSLB}

The \emph{initial forest} of a binary search tree $T$, denoted $\inc(T)$ or simply $\inc$ when there is no ambiguity, is a forest poset on the nodes of $T$ constructed by keeping only increasing relations of $T$, \emph{i.e.}:
\begin{equation*}
a \trprec_{\inc} b \text{ if, and only if } a < b ~\text{ and }~ a \trprec_{T} b.
\end{equation*}
It is equivalent to the following construction, for all nodes $x$ of $T$:
\begin{itemize}
  \item if $y$ is the left son of $x$ in $T$, then $y$ is the first son of $x$ in $\inc(T)$,
  \item if $y$ is the right son of $x$ in $T$, then $y$ is the right brother of $x$ in $\inc(T)$, \emph{i.e.} if $x$ is the $i^{th}$ son of its parent node in $F$ then $y$ is the $(i+1)^{th}$ son (we consider that all root nodes of $F$ have a common parent node). 
\end{itemize}
In the same way, one can define the \emph{final forest} (denoted
$\dec$) by reversing the roles of the right and left son in the previous
construction or, in terms of posets:
\begin{equation*}
b \trprec_{\dec} a \text{ if, and only if } b < a ~\text{ and }~ b \trprec_{T} a.
\end{equation*}
\end{Definition}

In our definition, $\inc$ and $\dec$ are directly defined as posets. They actually correspond
to some labelled planar forests. Indeed, our constructions are the translation of some well-known bijections between unlabelled binary trees and unlabelled planar forest. The labellings we obtain on $\dec$ and $\inc$ are canonical. For example, for $\inc$ it corresponds to the recursive traversal of each root from left to right: recursively label the the subtrees from left to right then label the root. As both are bijections, the tree $T$ can be recursively retrieved from one forest among $\inc(T)$ and $\dec(T)$. On the initial forest, the root of the corresponding binary tree is the left-most (minimal) root of the forest trees. The left subtree is obtained recursively from the subtrees of the root and the right subtree from the remaining trees of the forest. The construction is symmetric for the final forest. Both bijections are illustrated in Figure~\ref{fig:inc-dec}. As a convention, initial forest are always written in blue with trees oriented from left to right and final forests are in red with trees oriented from bottom to top.

\begin{figure}[ht]
  
  \begin{center}
    \scalebox{0.8}{
    \input{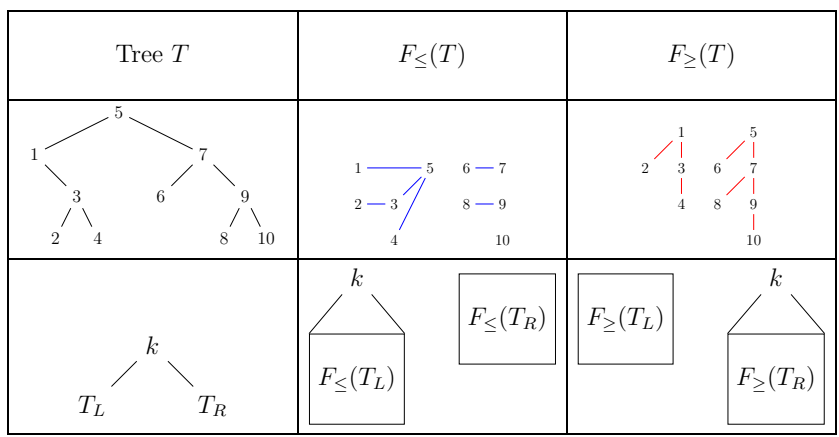}
    }
  \end{center}
  
  \caption{A tree with its corresponding initial and final forests.}
  
  \label{fig:inc-dec}
  
\end{figure}

Now, the question we want to answer is: if $P$ is a poset labelled with distinct integers $1, \dots, n$, is it the initial (or final) forest of some binary tree? The following lemma gives the the sufficient and necessary conditions for this to happen.



\begin{Lemme}
  \label{lem:carac-foret}
  Let $F$ be a labelled poset. Then $F$ is the initial forest of a binary tree $T$ if for every $a \trprec_F c$, we have $a < c$ and $b \trprec_F c$ for every $b$ such that $a < b <c$. Symmetrically, $F$ is the final forest of a binary tree $T$ if for every $c \trprec_F a$ we have $c > a$ and $b \trprec_F a$ for every $b$ such that $a < b <c$. 

\end{Lemme}

\begin{proof}
The final and initial forests cases are symmetric, we only give the proof of the final forest case. 

First let us proof that the final forest $F := \dec(T)$ of a binary tree $T$ satisfies the necessary condition. Let $c > a$ be such that $c \trprec_F a$. By construction of $F$, we also have $c \trprec_T a$ which means that $c$ is in the right subtree of $a$. Let $b$ be such that $a < b < c$. Only three configurations are possible: either $a$ is in the left subtree of $b$, or $a$ and $b$ are not comparable in $T$, or $b$ is in the right subtree of $a$. The first configuration never happens because it implies that $c$ is in the left subtree of $b$ which contradicts the binary search tree condition. Then, if $a$ and $b$ are not comparable, it means they have a common root $b'$ with $a < b' <b$ and $a$ is in the left subtree of $b'$. The situation is then similar to the previous one and leads to a contradiction as $c$ is also in the left subtree of $b'$. Only the third configuration is possible which makes $b \trprec_T a$ and by construction $b \trprec_F a$.

Now, let $F$ be a labelled poset satisfying the condition of the final forest. The poset $F$ is made of $r$ connected components $F_1, \dots F_r$. For each $F_i$, there is a unique minimal poset element $x_i$: $y \trprec_F x_i$ for all $y \in F_i$. We call it the root of $F_i$. Indeed, if $x$, $x'$, and $y$ are in $F_i$ with $y \trprec_F x$ and $y \trprec_F x'$ then either $x < x' < y$ and $x' \trprec_F x$ or $x' < x < y$ and $x \trprec_F x'$. As all relations of $F$ are decreasing relations, $x_i$ is also the minimal label of $F_i$: $y>x_i$ for all $y \in F_i$. Furthermore, if $x_i$ and $x_j$ are the roots of two different components $F_i$ and $F_j$ then $x_i < x_j$ implies $y < z$ for all $y \in F_i$ and $z \in F_j$. Now, following the construction described by Figure \ref{fig:inc-dec}, we set $k$ to be the maximal label among the roots $x_1, \dots, x_r$. If we cut out the root $k$ from its connected component, the remaining poset $F_L$ still satisfies the condition and all its labels are bigger than $k$. The poset $F_R$ made from the other remaining connected components also satisfy the condition and all its label are smaller than $k$. Then we can recursively construct the binary tree $T := k(T_L, T_R)$ where $T_L$ and $T_R$ are obtained from respectively $F_L$ and $F_R$. By construction, $T$ is a binary search tree and $F = \dec(T)$.
\end{proof}

We have seen that the linear extensions of a binary tree $T$ form an interval of the right weak order. The linear extensions of the initial and final forests of $T$ correspond to initial and final intervals \cite{BW} and can be interpreted in terms of the Tamari order.

\begin{Proposition}

\label{prop:minmax-interval}
The linear extensions of the initial forest $\inc(T)$ of a binary tree $T$ are the sylvester classes of 
 all trees $T' \leq T$ in the Tamari order
(initial interval) and the linear extensions of the final forest $\dec(T)$ of $T$ 
are the sylvester classes of all trees $T' \geq T$ (final
interval).

\end{Proposition}

\begin{proof}
We only give the proof for $\dec(T)$. By symmetry of the right weak order and the Tamari order, it also proves the result for $\inc(T)$. Let $\alpha_T$ be the minimal element of the sylvester class of $T$. We want to prove that the linear extensions of $\dec(T)$ are exactly the interval $\left[ \alpha_T, \omega \right]$ where $\omega$ is the maximal element of the right weak order. Since the Tamari order is a quotient of the right weak order, the Proposition is entirely proved by this result.

Let us recall that a coinversion $(a,b)$ of a permutation $\sigma$ is couple of numbers such that $a < b$ and $b$ appears before $a$ in $\sigma$. As an example, $(1,4)$ is a coinversion of $2431$ as well as $(1,2)$, $(3,4)$ and $(1,3)$. We have that $\mu \leq \sigma$ in the right weak order if and only if the coinversions of $\mu$ are contained in the coinversions of $\sigma$. For the previous example, the permutation $\mu = 2314$ is smaller than $\sigma$ because its coinversions $(1,2)$ and $(1,3)$ are also coinversions of $\sigma$.

The linear extensions of $\dec(T)$ are the permutations containing all coinversions $(a,b)$ where $b \trprec_{\dec} a$. It is clear by construction that a linear extension of $\dec(T)$ contains these coinversions. It is also a sufficient condition. Indeed let $\sigma$ be a permutation that is not a linear extension of $\dec(T)$. Then there is $(a,b)$ with $b \trprec_{\dec} a$ and $a$ before $b$ in $\sigma$. The permutation $\sigma$ does not contain the coinversion $(a,b)$. 

Finally, the permutation $\alpha_T$ contains exactly the coinversions given by the $\dec(T)$ relations (it does not contain other coinversions). Indeed, it is known \cite{PBT2} that $\alpha_T$ is read on the binary search tree by a recursive printing: left subtree, right subtree, root. Let $b > a$ be such that $\dec(T)$ does not contain the relation $b \trprec_{\dec} a$. It means $b$ is not on the right subtree of $a$. There are only two possible configurations: either $a$ is on the left subtree of $b$, either they have a common root $b'$ and $a$ is on the left subtree of $b'$ and $b$ on the right subtree of $b'$. In both cases, $a$ is read before $b$ in $\alpha_T$ and then $\alpha_T$ does not contain the coinversion $(a,b)$. 

To conclude, the linear extensions of $\dec(T)$ are the permutations whose coinversions contain the coinversions of $\alpha_T$. In other words, they are the permutations greater than or equal to $\alpha_T$.
\end{proof}

If two trees $T$ and $T'$ are such that $T \leq T'$, then $\dec(T)$
and $\inc(T')$ share some linear extensions (by
Proposition~\ref{prop:minmax-interval}). More precisely, we denote by $\ExtL(F)$ the set of linear extensions of a poset $F$. Then we have $\ExtL(\dec(T)) \cap \ExtL(\inc(T')) = [\alpha_T, \omega_{T'}]$ where $\alpha_T$ (resp. $\omega_{T'}$) is the minimal permutation (resp. maximal permutation) of the sylvester class of $T$ (resp. $T'$). This set corresponds exactly to the linear extensions of the trees of the interval $[T,T']$ in the Tamari order. It is then natural to construct a poset that would contain relations of both $\dec(T)$ and $\inc(T')$. That is what we call an \emph{interval-poset}. We give a first example in Figure \ref{fig:forest-intersection}. Note that unlike $\dec(T)$ and $\inc(T')$, the interval-poset formed by the reunion of their relations is not necessary a forest itself. The characterisation of interval-posets follows naturally from Lemma~\ref{lem:carac-foret}.

\begin{Definition}
  \label{def:interval-poset-definition}
  An \emph{interval-poset} $P$ is a poset labelled with distinct integers $\lbrace1, \dots, n \rbrace$ such that the following
  conditions hold:
  \begin{itemize}
    \item $a \trprec_{P} c$ and $a < c$ implies that for all $a < b < c$, we have
      $b \trprec_{P} c$, 
    \item $c \trprec_{P} a$ and $a < c$ implies that for all $a < b < c$,
      we have $b \trprec_{P} a$.
  \end{itemize}
\end{Definition}
As an example, poset

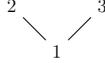
\begin{center}
\scalebox{0.6}{
\begin{tikzpicture}
\node(T1) at (0,0){1};
\node(T2) at (-1,1){2};
\node(T3) at (1,1){3};

\draw(T1) -- (T2);
\draw(T1) -- (T3);
\end{tikzpicture}
}
\end{center}
is not an interval-poset. Indeed, we have $1 \trprec 3$ without $2 \trprec 3$ so it does not satisfy the second condition of the definition. An example of an interval-poset is given in Figure \ref{fig:forest-intersection}. By convention, even though an interval-poset is by definition a poset we do not represent it by its Hasse diagram. For clarity, we draw the union of the Hasse diagrams formed respectively by its increasing relations (in blue) and decreasing relations (in red).

\begin{figure}[ht]
  \begin{center}
    
    \input{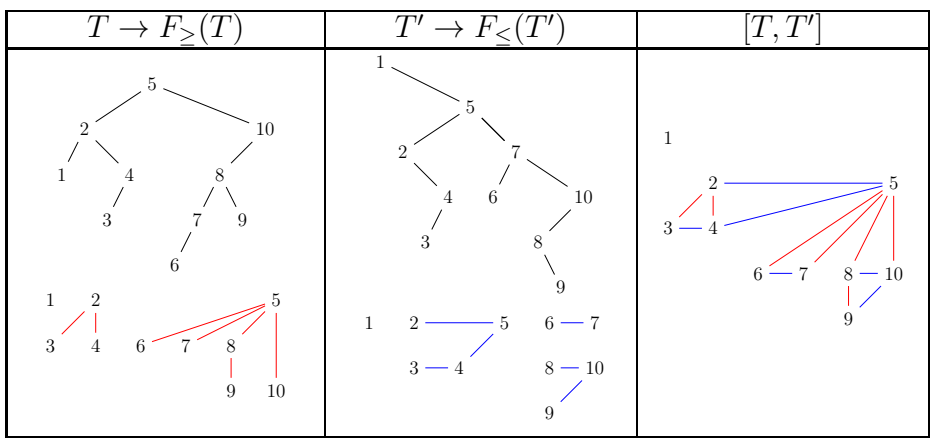}
    
  \end{center}
  
  \caption{Two trees $T$ and $T'$ with $T < T'$, their final and
    initial forest and the interval-poset $[T, T']$. This Tamari
    interval is shown in Figure
    \ref{fig:interval-forest-intersection}.}
  
  \label{fig:forest-intersection}
\end{figure}

\begin{figure}[p]
  \begin{center}
  \begin{fullpage}
    \input{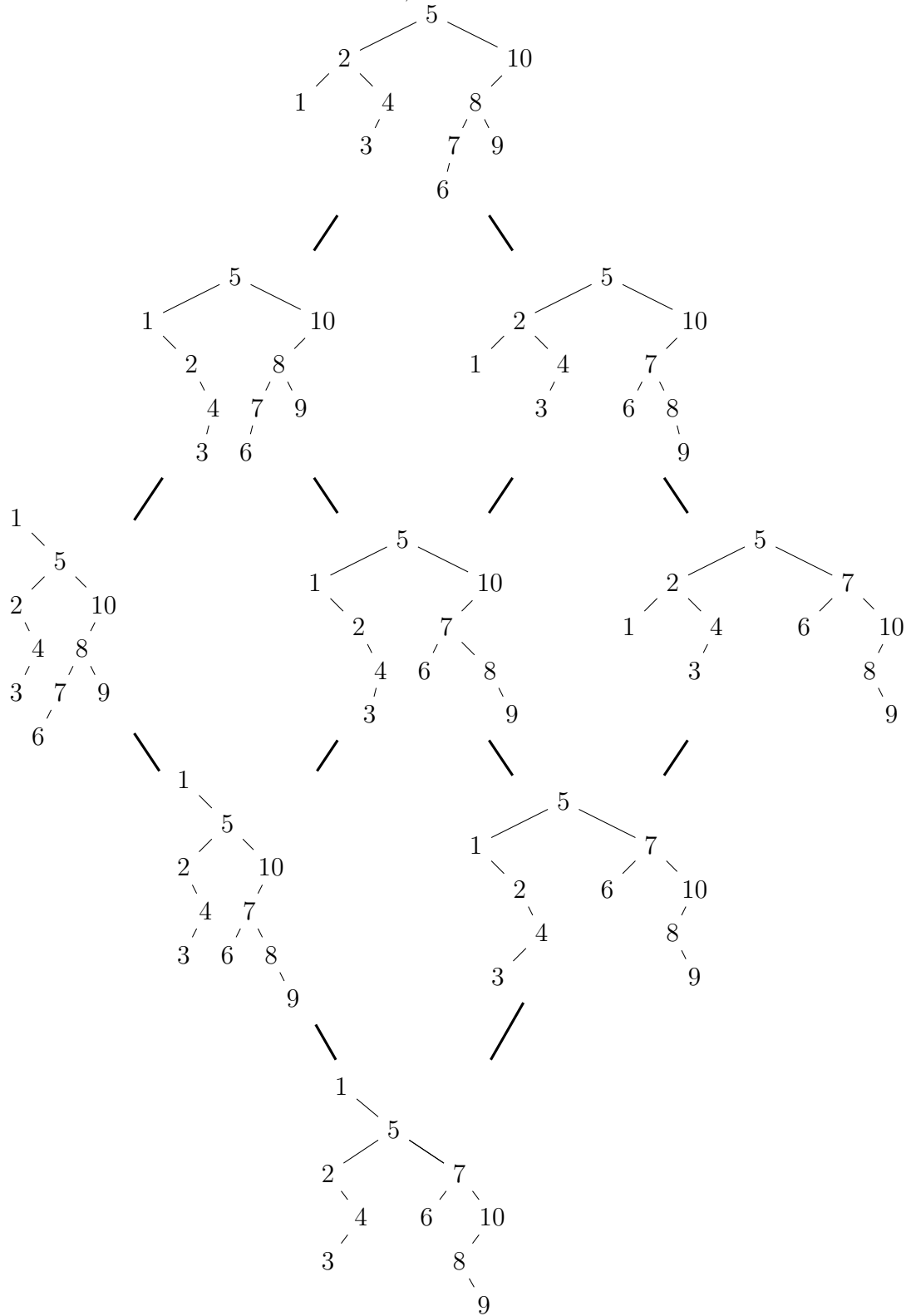}
\end{fullpage}
  \end{center}
  
  \caption{The interval between the trees $T$ and $T'$ of Figure~\ref{fig:forest-intersection}.}

  \label{fig:interval-forest-intersection}
\end{figure}

\begin{Theoreme}
  \label{prop:tamari-interval-characterization}
Interval-posets are in bijection with intervals of the Tamari order.

More precisely, to each interval-poset corresponds a couple of binary trees $T_1 \leq T_2$ such that the linear extensions of the interval-poset are exactly the linear extensions of the binary trees $T' \in [T_1, T_2]$.

And conversely, interval-posets are the only labelled posets whose linear extensions are intervals of the right weak order $[\alpha_{T_1}, \omega_{T_2}]$ with $\alpha_{T_1}$ (resp. $\omega_{T_2}$) the minimal permutation (resp. maximal permutation) of a sylvester class.
\end{Theoreme}

\begin{proof}
Let $[T_1, T_2]$ be an interval of the Tamari order. We build a poset containing all the relations from both $\dec(T_1)$ and $\inc(T_2)$. Note that relations from $\dec(T_1)$ and $\inc(T_2)$ together can never produce a cycle. Indeed any linear extension of $T_1$ for example satisfies both by Proposition \ref{prop:minmax-interval}. It is clear by Lemma \ref{lem:carac-foret} that the resulting poset is an interval-poset.

Conversely, from an interval-poset $P$, we build $\dec$ and $\inc$ by keeping respectively decreasing and increasing relations of $P$. By Lemma \ref{lem:carac-foret}, the two resulting posets are respectively a final forest of a binary tree $T_1$ and an initial forest of a binary tree $T_2$. Let $\sigma$ be a linear extension of $P$ whose sylvester class corresponds to a binary tree $T'$. By definition, the permutation $\sigma$ is also a linear extension of $\dec$ and $\inc$ and we have by Proposition \ref{prop:minmax-interval} that $T_1 \leq T' \leq T_2$. As $T_1 \leq T_2$, the interval $[T_1,T_2]$ is well defined. 
\end{proof}

Many operations on intervals can be easily adapted on interval-posets, all with trivial proofs.

\begin{Proposition}
\label{prop:comb-prop-trivial}
\begin{enumerate}[label=(\roman{*}), ref=(\roman{*})]
\item The intersection between two intervals $I_1$ and $I_2$ is given by the union of their relations $I_3$. If $I_3$ is a valid poset, \emph{i.e.}, there is contradictions between the relations of $I_1$ and $I_2$, then $I_3$ is an interval-poset, otherwise, the intersection is empty. 
\label{prop:comb-prop-intersect}
\item An interval $I_1 := \left[ T_1, T_1' \right] $ is contained in an interval $I_2 := \left[ T_2, T_2' \right]$, \emph{i.e.}, $T_1 \geq T_2$ and $T_1' \leq T_2'$, if and only if all relations of the interval-poset $I_1$ are satisfied by the interval-poset $I_2$. 
\label{prop:comb-prop-inclusion}
\item If $I_1 := \left[ T_1, T_1' \right]$ is an interval, then $I_2 = \left[ T_2, T_1' \right]$, $T_2 \geq T_1$, if and only if all relations of the interval-poset $I_1$ are satisfied by $I_2$ and all new relations of $I_2$ are decreasing. Symmetrically, $I_3 = \left[ T_1, T_3 \right]$, $T_3 \leq T_1'$, if and only if all relations of the interval-poset $I_1$ are satisfied by $I_3$ and all new relations of $I_3$ are increasing. 
\label{prop:comb-prop-minmax-inclusion}
\end{enumerate}
\end{Proposition}

The interval-poset consisting of the set of points $1, \dots, n$ without any relations corresponds to the whole Tamari lattice and the linear extensions are all permutations of size $n$. More generally, let $I$ be the interval-poset of an interval $[T,T']$. If $\tilde{T}$ is the binary tree $T$ where a right rotation has been applied on a couple of nodes $x < y$, then $\tilde{I}$ is the interval-poset $I$ where the decreasing relation $y \trprec x$ has been added. Symmetrically, applying a left rotation on a couple of nodes $y > x$ on $T'$ corresponds to adding the increasing relation $x \trprec y$. This is illustrated in Figure \ref{fig:interval-poset-rotation}.

\begin{figure}[ht]
\centering
\scalebox{0.8}{
\input{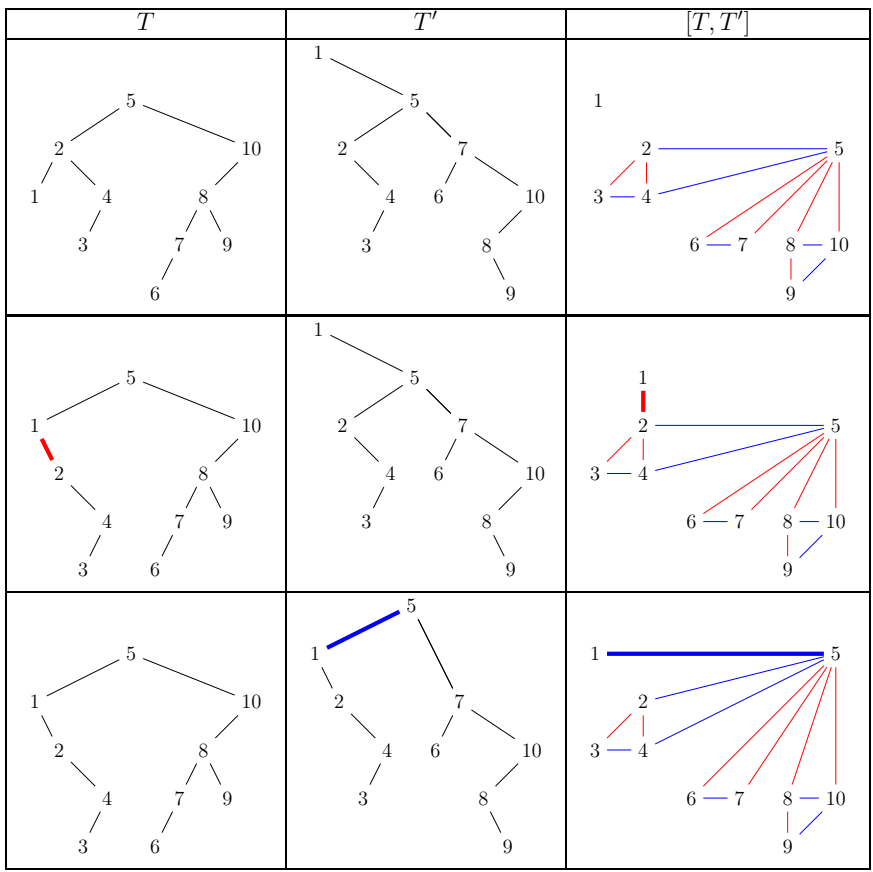}
}
\caption[Adding relations on interval-posets]{Adding relations on interval-posets: adding a decreasing relation to the poset makes a right rotation on $T$ and adding an increasing relation makes a left rotation on $T'$.}
\label{fig:interval-poset-rotation}
\end{figure}

As posets satisfying simple properties, interval-posets are easy to implement in any computer algebra system. Thereby they allow for computer exploration on the combinatorics of Tamari intervals. In this purpose, we decided to integrate them into the mathematical software Sage\cite{SAGE_WEBSITE}. They were developed in \cite{SAGE_IntervalPosets} and are available since {\tt Sage~6.3}. Especially, all basic operations such as the ones of Proposition~\ref{prop:comb-prop-trivial} are implemented. The details of the different functionalities are given in the documentation \cite{SAGE_IntervalPosets}, we give some basic examples in Appendix~\ref{app:sub-sec:basic}.

\section{Tamari polynomials}
\label{sec:tamari-polynomials}

\subsection{Composition of interval-posets}
\label{sec:bilinear-form}

Let $\phi(y)$ be the generating function of Tamari intervals,
\begin{equation}
\phi(y) := \sum_{n \geq 0} I_n y^n
\end{equation}
where $I_n$ is the number of intervals of trees of size $n$, equivalently this is the number of interval-posets with $n$ vertices. The first values of $I_n$ are given in \cite{OEISIntervalles}
\begin{equation}
\label{eq:generating-function}
\phi(y) = 1 + y + 3 y^2 + 13 y^3 + 68 y^4 + \cdots~.
\end{equation}

In \cite{Chap}, Chapoton gives a refined version of $\phi$,
\begin{equation}
\Phi(x,y) := \sum_{n,m \geq 0}I_{n,m} x^m y^n
\end{equation}
where $I_{n,m}$ is the number of intervals $[T_1, T_2]$ of trees of size $n$ such that $T_1$ has exactly $m$ nodes on its leftmost branch. This gives
\begin{equation}
\label{eq:refined-generating-function}
\Phi(x,y) = 1 + xy + (x + 2x^2)y^2 + (3x + 5x^2 + 5x^3)y^3 + \cdots~.
\end{equation}

The statistic of the number of nodes on the leftmost branch is well known \cite{FindStatLeftBranch}. On Dyck paths, it corresponds to the number of touch points: the number of contacts between the path and the bottom line \cite{FindStatTouchPoints}. It can also be read on $\dec(T)$: it is the number of connected components.

\begin{Definition}
\label{def:stat}
Let $[T_1, T_2]$ be an interval and $I$ its interval-poset, we denote by
\begin{enumerate}
\item $\Isize(I)$ the number of vertices of $I$, \emph{i.e.}, the size of the trees $T_1$ and $T_2$.
\item $\Itrees(I)$ the number of connected components (or trees) of $\dec(I)$, the poset obtained by keeping only decreasing relations of $I$. 
\end{enumerate}
We then define $\PI(I) := x^{\Itrees(I)}y^{\Isize(I)}$ which we extend by linearity to linear combinations of interval-posets.
\end{Definition}

The refined generating functions $\Phi$ can be expressed as
\begin{equation}
\Phi(x,y) = \sum_{I} \PI(I)
\end{equation}
summed on all interval-posets. We prove the following theorem.

\begin{Theoreme}
\label{thm:functional-equation}
The generating function $\Phi(x,y)$ satisfies the functional equation
\begin{equation}
\label{eq:functional-equation}
\Phi(x,y) = \B(\Phi,\Phi) + 1
\end{equation}
where
\begin{equation}
\label{eq:bilinear}
\B(f,g) := xy f(x,y) \frac{x g(x,y) - g(1, y)}{x - 1}.
\end{equation}
\end{Theoreme}

This theorem was already proven by Chapoton in \cite{Chap}\footnote{Our equation is slightly different from the one of \cite[formula (6)]{Chap}. Indeed, the definition of the degree of $x$ differs by one and in our case $\Phi$ also counts the interval of size $0$.}. In Section~\ref{sec:enumeration}, we give a new proof of the theorem based on a combinatorial interpretation of the operator $\B$. Let us define now what we call the \emph{composition} of interval-posets.

\begin{Definition}
\label{def:composition}
Let $I_1$ and $I_2$ be two interval-posets of size respectively $k_1$ and $k_2$. Then $ \BB(I_1, I_2)$ is the formal sum of all interval-posets of size $k_1 + k_2 + 1$ where, 
\begin{enumerate}[label=(\roman{*}), ref=(\roman{*})]
\item the relations between vertices $1, \dots, k_1$ are exactly the ones from $I_1$,
\label{def:composition:cond:I1}
\item the relations between $k_1 + 2, \dots, k_1 + k_2 + 1$ are
  exactly the ones from $I_2$ shifted by $k_1 + 1$,
\label{def:composition:cond:I2}
\item we have $i \trprec k_1 +1$ for all $i \leq k_1$,
\label{def:composition:cond:incr}
\item there is no relation $k_1+1 \trprec j$ for all $j>k_1+1$.
\label{def:composition:cond:decr}
\end{enumerate}
We call this operation the composition of intervals and extend it by bilinearity to all linear sums of intervals.
\end{Definition}

\begin{figure}[ht]

  \begin{center}

    \input{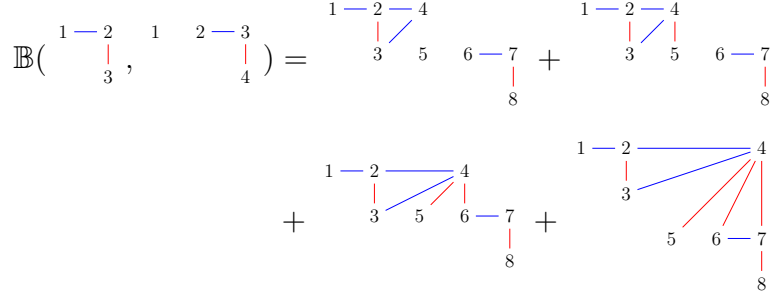}

    \caption{Composition of interval-posets: the four terms of the
      sum are obtained by adding respectively no, 1, 2, and 3 decreasing
      relations between the second poset and the vertex $4$.}

    \label{fig:composition}
    
  \end{center}
  
\end{figure}

The sum we obtain by composing interval-posets actually corresponds to all possible ways of adding decreasing relations between the second poset and the new vertex $k_1 + 1$, as seen in Figure \ref{fig:composition}. Especially, there is no relations between vertices $1, \dots, k_1$ and $k_1+2, \dots, k_1+k_2+1$. Indeed, condition \ref{def:composition:cond:incr} makes it impossible to have any relation $j \trprec i$ with $i<k_1+1<j$ as this would imply by Definition \ref{def:interval-poset-definition} that $k_1+1 \trprec i$. And condition \ref{def:composition:cond:decr} makes it impossible to have $i \trprec j$ as this would imply $k_1 + 1 \trprec j$.

The number of elements in the sum is given by $\Itrees(I_2) + 1$. Indeed, if $x_1 < x_2 < \dots < x_m$ are the tree roots of $\dec(I_2)$, a decreasing relation $x_i \trprec k_1 +1$ can be added only if all relations $x_j \trprec k_1+1$ for $j<i$ have already been added. We then obtain
\begin{equation}
\label{eq:simple-composition}
\BB(I_1,I_2) = \sum_{0 \leq i \leq m} P_i
\end{equation}
where $P_i$ is the interval-poset where exactly $i$ relations have been added: $x_j \trprec k_1 +1$ for $j \leq i$.

\begin{figure}[ht]
\centering
\scalebox{0.8}{

\def \fpath{figures/}

\begin{tabular}{|c|c|}
\hline
Interval-poset & Corresponding interval \\
\hline

\begin{tikzpicture}[scale=1]
\node(T1) at (0,0) {1};
\node(T2) at (1,0) {2};
\node(T3) at (1,-1) {3};
\draw[line width = 0.5, color=blue] (T1) -- (T2);
\draw[line width = 0.5, color=red] (T3) -- (T2);
\end{tikzpicture}

&
\scalebox{0.6}{

%

\begin{tikzpicture}[scale=0.8]
\node (N0) at (1.500, 0.000){2};
\node (N00) at (0.500, -1.000){1};
\node (N01) at (2.500, -1.000){3};
\draw (N0) -- (N00);
\draw (N0) -- (N01);
\end{tikzpicture}},\scalebox{0.6}{}

\\
\hline

\begin{tikzpicture}[scale=1]
\node(T1) at (-1,-1) {1};
\node(T2) at (0,-1) {2};
\node(T3) at (1,-1) {3};
\node(T4) at (1,-2) {4};
\draw[line width = 0.5, color=blue] (T2) -- (T3);
\draw[line width = 0.5, color=red] (T4) -- (T3);
\end{tikzpicture}

&
\scalebox{0.6}{


\begin{tikzpicture}[scale=0.8]
\node (N0) at (2.500, 0.000){3};
\node (N00) at (1.500, -1.000){2};
\node (N000) at (0.500, -2.000){1};
\draw (N00) -- (N000);
\node (N01) at (3.500, -1.000){4};
\draw (N0) -- (N00);
\draw (N0) -- (N01);
\end{tikzpicture}},\scalebox{0.6}{

%

\begin{tikzpicture}[scale=0.8]
\node (N0) at (0.500, 0.000){1};
\node (N01) at (2.500, -1.000){3};
\node (N010) at (1.500, -2.000){2};
\node (N011) at (3.500, -2.000){4};
\draw (N01) -- (N010);
\draw (N01) -- (N011);
\draw (N0) -- (N01);
\end{tikzpicture}} 
\\
\hline
\input{\fpath interval-posets/I8-ex1}
&
\scalebox{0.6}{\input{\fpath trees/T8-ex1-n}},\scalebox{0.6}{\input{\fpath trees/T8-ex2-n}}
\\
\hline
\input{\fpath interval-posets/I8-ex2}
&
\scalebox{0.6}{\input{\fpath trees/T8-ex3-n}},\scalebox{0.6}{\input{\fpath trees/T8-ex2-n}}
\\
\hline
\input{\fpath interval-posets/I8-ex3}
&
\scalebox{0.6}{\input{\fpath trees/T8-ex4-n}},\scalebox{0.6}{\input{\fpath trees/T8-ex2-n}}
\\
\hline
\input{\fpath interval-posets/I8-ex4}
&
\scalebox{0.6}{\input{\fpath trees/T8-ex5-n}},\scalebox{0.6}{\input{\fpath trees/T8-ex2-n}}
\\
\hline

\end{tabular}
}
\caption{Interval interpretation of the composition of interval-posets.}
\label{fig:composition-2}
\end{figure}

\begin{Proposition}
Let $I_1$ be an interval-poset of size $k_1$ with $[T_1, T_1']$ as corresponding interval and $I_2$ an interval-poset of size $k_2$ with $[T_2,T_2']$ as corresponding interval. We set $k:= k_1 + 1$ and

\begin{enumerate}
\item $Q_\alpha$, the binary tree obtained by grafting $k(T_1, \emptyset)$ on the left of the leftmost node of $T_2$,
\item $Q_\omega$, the binary tree $k(T_1, T_2)$,
\item $Q'$, the binary tree $k(T_1',T_2')$.
\end{enumerate}
Then we have
\begin{equation}
\BB(I_1, I_2) = \sum_{Q \in [Q_\alpha,Q_\omega]} P_{[Q,Q']}
\end{equation}
where $P_{[Q,Q']}$ is the interval-poset of $[Q,Q']$.
\end{Proposition}

\begin{proof}
The composition of $I_1$ and $I_2$ is a sum of interval-posets $P_0, \dots, P_{m}$ where $m = \Itrees(I_2)$ and where $P_i$ is the interval-poset where exactly $i$ decreasing relations have been added. The maximal tree of all intervals is always the same as they all share the same increasing relations. This maximal tree is $Q':=k(T_1', T_2')$. The final forest $\dec(P_0)$ of $P_0$ contains $\Itrees(I_1) + \Itrees(I_2) +1$ connected components. The nodes on the leftmost branch of its minimal tree are given by those of $T_1$, then $k$, then those of $T_2$, \emph{i.e.} this is exactly $Q_\alpha$. Let $Q_i$ be the minimal tree of $P_i$. To go from $P_i$ to $P_{i+1}$, a decreasing relation is added to the vertex $k$: this corresponds to a rotation between the node $k$ of $Q_i$ and its parent node. This process ends when $T_2$ has been completely switched to
the right side of the node $k$. We then obtain the tree $Q_m = Q_\omega$.

The interval between $Q_\alpha$ and $Q_\omega$ is actually a saturated chain: $Q_\alpha = Q_0 \lessdot Q_1 \lessdot \dots \lessdot Q_m = Q_\omega$.
\end{proof}

As an example, the interpretation of the computation in Figure \ref{fig:composition} in terms of intervals is given in Figure \ref{fig:composition-2}. 

The composition of intervals-posets can be decomposed into two different operations: a left product $\pleft$ and a right product $\pright$.

\begin{Definition}
\label{def:left-right-product}
Let $I_1$ and $I_2$ be two interval-posets such that $\Itrees(I_2) = m$ with $x_1 < x_2 < \dots < x_m$ the roots of the trees of $\dec(I_2)$. Let $\alpha$ and $\omega$ be respectively the label of minimal value of $I_2$ and the label of maximal value of $I_1$. Then
\begin{enumerate}
\item $I_1 \pleft I_2$ is the interval-poset obtained by a shifted concatenation of $I_1$ and $I_2$ with the increasing relations $y \trprec \alpha$ for all $y \in I_1$.
\item $I_1 \pright I_2$ is the sum of the $m+1$ interval-posets $P_0, P_1, \dots, P_m$ where $P_i$ is the shifted concatenation of $I_1$ and $I_2$ with exactly $i$ added decreasing relations: $x_j \trprec \omega$ for $j \leq i$.
\end{enumerate}
\end{Definition}

As an example,
\begin{align}
\notag
\begin{aligned}\scalebox{0.8}{

\begin{tikzpicture}[xscale=0.8, yscale=0.8]
\node(T2) at (0,-1) {2};
\node(T1) at (0,0) {1};
\node(T3) at (1,-1) {3};
\draw[line width = 0.5, color=blue] (T2) -- (T3);
\draw[line width = 0.5, color=red] (T2) -- (T1);
\end{tikzpicture}
}\end{aligned}
\pleft
\begin{aligned}\scalebox{0.8}{

\begin{tikzpicture}[xscale=0.5, yscale=0.8]
\node(T1) at (0,0) {1};
\node(T2) at (1,0) {2};
\node(T3) at (1,-1) {3};
\draw[line width = 0.5, color=red] (T3) -- (T2);
\end{tikzpicture}
}\end{aligned} &=
\begin{aligned}\scalebox{0.8}{\input{figures/interval-posets/I6-ex2}}\end{aligned} \\
\notag
\begin{aligned}\scalebox{0.8}{

\begin{tikzpicture}[xscale=0.8, yscale=0.8]
\node(T2) at (0,-1) {2};
\node(T1) at (0,0) {1};
\node(T3) at (1,-1) {3};
\draw[line width = 0.5, color=blue] (T2) -- (T3);
\draw[line width = 0.5, color=red] (T2) -- (T1);
\end{tikzpicture}
}\end{aligned}
\pright
\begin{aligned}\scalebox{0.8}{

\begin{tikzpicture}[xscale=0.5, yscale=0.8]
\node(T1) at (0,0) {1};
\node(T2) at (1,0) {2};
\node(T3) at (1,-1) {3};
\draw[line width = 0.5, color=red] (T3) -- (T2);
\end{tikzpicture}
}\end{aligned} &=
\begin{aligned}\scalebox{0.8}{

\begin{tikzpicture}[scale=0.8]
\node(T2) at (0,-1) {2};
\node(T1) at (0,0) {1};
\node(T3) at (1,-1) {3};
\node(T4) at (1,-2) {4};
\node(T5) at (2,-2) {5};
\node(T6) at (2,-3) {6};
\draw[line width = 0.5, color=blue] (T2) -- (T3);
\draw[line width = 0.5, color=red] (T2) -- (T1);
\draw[line width = 0.5, color=red] (T6) -- (T5);
\end{tikzpicture}
}\end{aligned} +
\begin{aligned}\scalebox{0.8}{\input{figures/interval-posets/I6-ex4}}\end{aligned} +
\begin{aligned}\scalebox{0.8}{\input{figures/interval-posets/I6-ex5}}\end{aligned}.
\end{align}

From the description of the composition given by \eqref{eq:simple-composition}, it is clear that
\begin{equation}
\BB(I_1, I_2) = I_1~\pleft~u~\pright~I_2
\end{equation}
where $u$ is the interval-poset with a single vertex. For example, the composition of Figure~\ref{fig:composition} reads
\begin{align*}
\BB( 
\begin{aligned}
\scalebox{0.6}{}
\end{aligned}
,
\begin{aligned}
\scalebox{0.6}{}
\end{aligned}
) &= 
\begin{aligned}
\scalebox{0.6}{}
\end{aligned}
~\pleft~ 
\scalebox{0.6}{1}
~\pright~
\begin{aligned}
\scalebox{0.6}{}
\end{aligned}.
\end{align*}

The order on the two operations is not important: $(I_1 \pleft u) \pright I_2 = I_1 \pleft (u \pright I_2)$. In Appendix \ref{app:sub-sec:comp}, we give the Sage code of the composition using left and right products.

\subsection{Enumeration of interval-posets}
\label{sec:enumeration}

The $\B$ operator can also be decomposed into two operations,
\begin{equation}
\label{eq:polleft-right}
f \polleft g := fg, \qquad 
f \polright g := f \Delta(g),
\end{equation}
where
\begin{equation}
\label{eq:delta}
\Delta(g) := \frac{x g(x,y) - g(1,y)}{x - 1}.
\end{equation}
And we have
\begin{equation}
\B(f,g) = f \polleft xy \polright g.
\end{equation}

The composition of interval-posets is a combinatorial interpretation of the $\B$ operator as stated in the following Proposition.

\begin{Proposition}
\label{prop:combinatorial-equivalence-composition}
Let $I_1$ and $I_2$ be two interval-posets and $\PI$ the linear map from Definition \ref{def:stat}. Then
\begin{align}
\label{eq:equiv-pleft}
\PI(I_1 \pleft I_2) &= \PI(I_1) \polleft \PI(I_2), \\
\label{eq:equiv-pright}
\PI(I_1 \pright I_2) &= \PI(I_1) \polright \PI(I_2),
\end{align}
and consequently
\begin{equation}
\PI(\BB( I_1,  I_2)) = \B(\PI(I_1), \PI(I_2)).
\end{equation}
\end{Proposition}

As an example, in Figure \ref{fig:composition}, $\PI(I_1) = x^2 y^3$ and $\PI(I_2) = x^3 y^4$. And we have $\PI(\BB(I_1,I_2)) = y^8(x^6 + x^5 + x^4 + x^3) = \B(x^2 y ^3, x^3 y^4)$.

\begin{proof}
Let $I_1$ and $I_2$ be two interval-posets. The left product $I_1 \pleft I_2$ is the shifted concatenation of $I_1$ and $I_2$ on which only increasing relations have been added. Clearly,
\begin{equation*}
\PI(I_1 \pleft I_2) = y^{\Isize(I_1) + \Isize(I_2)} x^{\Itrees(I_1) + \Itrees(I_2)} = \PI(I_1)\PI(I_2) 
\end{equation*}
which proves \eqref{eq:equiv-pleft}.

Now, let $I_2$ be such that $\Itrees(I_2) = m$ and let $x_1 < x_2 < \dots < x_m$ be the roots of $\dec(I_2)$. By definition,
\begin{equation*}
I_1 \pright I_2 = \sum_{0 \leq i \leq m} P_i
\end{equation*}
where exactly $i$ decreasing relations have been added between roots $x_1, \dots, x_i$ of $\dec(I_2)$ and the vertex with maximal label of $I_1$. We have $\Itrees(P_i) = \Itrees(I_1) + \Itrees(I_2) - i$ because each added decreasing relation reduces the number of trees by one. Then
\begin{align*}
\PI(I_1 \pright I_2) &= y^{\Isize(I_1) + \Isize(I_2)}x^{\Itrees(I_1)}(1 + x + x^2 + \dots x^m) \\
&= y^{\Isize(I_1) + \Isize(I_2)}x^{\Itrees(I_1)} \frac{x^{m+1} - 1}{x - 1} \\
&= \PI(I_1) \polright \PI(I_2).
\qedhere
\end{align*}
\end{proof}

Now, to prove Theorem \ref{thm:functional-equation}, we only need the following proposition.

\begin{Proposition}
\label{prop:unicity-composition}
Let $I$ be an interval-poset, then, there is exactly one pair of intervals $I_1$ and $I_2$ such that $I$ appears in the composition $ \BB( I_1, I_2)$. 
\end{Proposition}

\begin{proof}
Let $I$ be an interval-poset of size $n$ and let $k$ be the vertex of $I$ with maximal label such that $i \trprec k$ for all $i < k$. The vertex $1$ satisfies this property, so one can always find such a vertex. We prove that $I$ only appears in the composition of $I_1$ by $I_2$, where $I_1$ is formed by the vertices and relations of $1, \dots, k -1$ and $I_2$ is formed by the re-normalized vertices and relations of $k + 1, \dots, n$. Note that one or both of these intervals can be of size 0.

Conditions \ref{def:composition:cond:I1}, \ref{def:composition:cond:I2}, and \ref{def:composition:cond:incr} of Definition \ref{def:composition} are clearly satisfied by construction. If condition \ref{def:composition:cond:decr} is not satisfied, it means that we have a relation $k \trprec j$ with $j>k$. Then, by definition of an interval-poset, we also have $\ell \trprec j$ for all $k<l<j$ and by definition of $k$, we have $i\trprec k \trprec j$ for all $i<k$, so for all $i<j$, we have $i \trprec j$. This is not possible as $k$ has been chosen to be maximal among vertices with this property. 

This proves that $I$ appears in the composition of $I_1$ by $I_2$. Now, if $I$ appears in $\BB(I_1',I_2')$, the vertex $k' := \vert I_1 ' \vert + 1$ is by definition the vertex where for all $i<k'$, we have $i \trprec k'$ and for all $j>k'$, we have $k' \nprec j$, this is exactly the definition of $k$. So $k' = k$ which makes $I_1' = I_1$ and $I_2' = I_2$.
\end{proof}

\begin{proof}[Proof of Theorem \ref{thm:functional-equation}]
Let $\SI := \sum_{T_1 \leq T_2} P_{[T_1,T_2]}$ be the formal power series of interval-posets. By Proposition \ref{prop:unicity-composition}, we have
\begin{equation*}
\SI = \BB(\SI,\SI) + \emptyset.
\end{equation*}
And by Proposition \ref{prop:combinatorial-equivalence-composition}, we have
\begin{align*}
\Phi& = \PI(\SI) 
= \PI(\BB(\SI,\SI)) + 1 \\
&= \B(\Phi,\Phi) + 1.
\qedhere
\end{align*}
\end{proof}

\subsection{Counting smaller elements in Tamari}
\label{sec:main-result}

By developing \eqref{eq:functional-equation}, we obtain
\begin{align*}
\Phi &= 1 + \B(1,1) + \B(\B(1,1),1) + \B(1,\B(1,1)) + \cdots \\
&= \sum_T y^{|T|}\BT_T,
\end{align*}
where $\BT_T$ is the Tamari polynomial of Definition \ref{def:tamari-polynomials}. Theorem \ref{thm:smaller-trees} is proved by the following proposition.

\begin{Proposition}
\label{prop:sum-composition}
Let $T := k(T_L, T_R)$ be a binary tree and $S_T := \sum_{T' \leq T} P_{[T',T]}$ the sum of all interval-posets whose maximal tree is $T$. Then $S_T = \BB(S_{T_L}, S_{T_R})$.
\end{Proposition}

\begin{proof}
Let $T$ be a binary tree of size $n$ such that $T = k(T_L,T_R)$. The interval-poset of the initial interval $[T_0,T]$ is $\inc(T)$, the initial forest of $T$. From Proposition \ref{prop:comb-prop-trivial} \ref{prop:comb-prop-minmax-inclusion}, the sum $S_T$ is the sum of all interval-posets $I$ which extends $\inc(T)$ by adding only decreasing relations.

Let $I$ be an interval of $S_T$. Let $I_L$ and $I_R$ be the sub-posets obtained by restricting $I$ to respectively $1, \dots, k-1$ and $k+1, \dots, n$. By the recursive definition of initial forests given by Figure \ref{fig:inc-dec}, $I_L$ and $I_R$ are poset extensions of respectively $\inc(T_L)$ and $\inc(T_R)$ where only decreasing relations have been added. And then $I_L \in S_{T_L}$ and $I_R \in S_{T_R}$. Finally, it is clear that $I \in \BB(I_L,I_R)$. Indeed, $I$ is a poset extension of $\inc(T)$ and so $i \trprec k$ for $i < k$ and $k \ntrprec j$ for $j > k$.

Conversely, if $I_L$ and $I_R$ are elements of respectively $S_{T_L}$ and $S_{T_R}$, then any interval $I$ of $\BB(I_L,I_R)$ is in $S_T$. Indeed, by construction, $I$ is an extension of $\inc(T)$ where only decreasing relations have been added.
\end{proof}

For a given tree $T$, the coefficient of the monomial with maximal degree in $x$ in $\BT_T$ is always 1. It corresponds to the initial interval $\inc(T)$. The interval with the maximal number of decreasing relations corresponds to $[T,T]$. An example of $\BT_T$ and of the computation of smaller trees is presented in Figure \ref{fig:BTExample}.

\begin{figure}[ht]

  \input{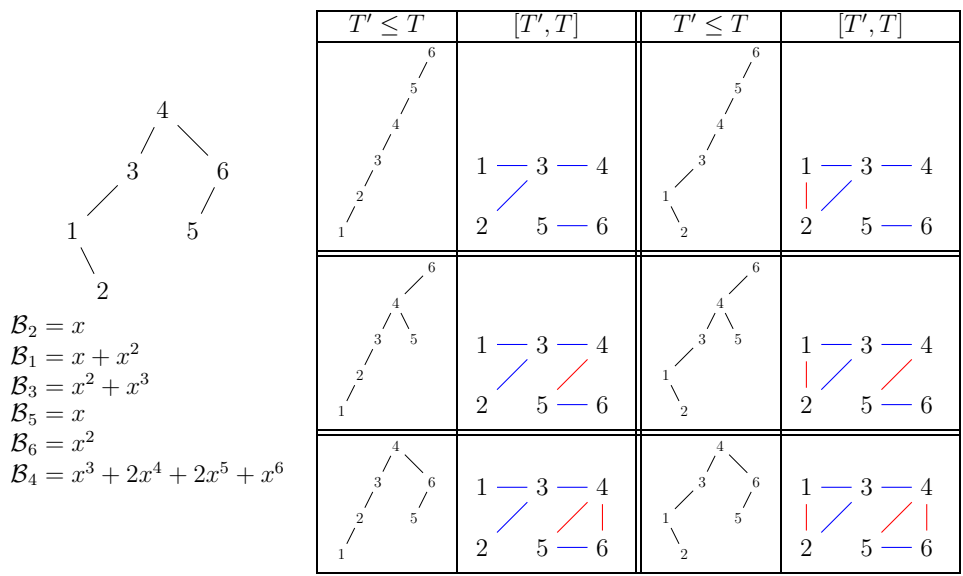}

  \caption{Example of the computation of $\BT_T$ and list of all
    smaller trees with associated intervals.}
  \label{fig:BTExample}
\end{figure}

\begin{proof}[Proof of Theorem \ref{thm:smaller-trees}]
Counting the number of trees $T' \leq T$ refined by the number of nodes on their leftmost branch can be done by counting the number of intervals $I = [T',T]$ refined by $\Itrees(T')$. We then want to prove that $\BT_T = \PI(S_T)$ where $S_T = \sum_{T' \leq T} P_{[T',T]}$. It can be done by induction on the size of $T$. The initial case is trivial. And if we set $T=k(T_L,T_R)$, by the induction hypothesis, we have that $\BT_{T_L} = \PI(S_{T_L})$ and $\BT_{T_R} = \PI(S_{T_R})$. The result is then a direct consequence of Propositions \ref{prop:combinatorial-equivalence-composition} and \ref{prop:sum-composition},
\begin{align*}
\BT_T &= \B(\PI(S_{T_L}),\PI(S_{T_R})) 
= \PI(\BB(S_{T_L},S_{T_R})) \\
&= \PI(S_T).
\qedhere
\end{align*}
\end{proof}

\subsection{Bivariate polynomials}
\label{sec:bivar}

In some recent work \cite{ChapBiVar}, Chapoton computed some bivariate polynomials that seem to be similar to the ones we study. The computation  of the first examples of \cite[formula (7)]{ChapBiVar} with  $b=1$ and $t = 1 - 1/x$ leads us to conjecture that each polynomial is equal to some $\BT_T(x)$ where $T$ is a binary tree with no left subtree. More precisely, the non planar rooted tree corresponding to $T$ is the non planar version of the planar forest $\dec(T)$.

A $b$ parameter can be also be added to our formula. For an interval $[T',T]$, it is either the number of nodes in $T'$ which have a right subtree, or in the interval-poset the number of nodes $x$ with a relation $y \trprec x$ and $y>x$. By a generalization of the linear function $\PI$, one can associate a monomial in $b$, $x$, and $y$ with each interval-poset. The bilinear form now reads:
\begin{equation}
\B(f,g) = y \left( xbf \frac{x g - g_{x=1}}{x - 1} -bxfg + xfg \right),
\end{equation}
where $f$ and $g$ are polynomials in $x$, $b$, and $y$. Proposition~\ref{prop:combinatorial-equivalence-composition} still holds, since a node with a decreasing relation is added in all terms of the composition but one. As an example, in Figure \ref{fig:composition}, one has $\B(y^3x^2b,y^4x^3b) = y^8(x^6b^2 + x^5b^3 + x^4b^3 + x^3b^3)$.

With this definition of the parameter $b$, the bivariate polynomials $\BT 
_T(x,b)$ where $T$ has no left subtree seem to be exactly the ones computed by Chapoton in \cite{ChapBiVar} when taken on $t = 1 - 1/x$. This seems to indicate some combinatorial and algebraic links between structures from very different mathematical contexts.

\section{$m$-Tamari lattices}
\label{sec:mTamari}

\subsection{Definition}
\label{sec:m-tam-def}

The $m$-Tamari lattices are a generalization of the Tamari lattice where objects have a $(m+1)$-ary structure instead of binary. They were introduced in \cite{BergmTamari} and can be described in terms of $m$-ballot paths. A $m$-ballot path is a lattice path from $(0,0)$ to $(nm,n)$ made from horizontal steps $(1,0)$ and vertical steps $(0,1)$ which always stays above the line $y=\frac{x}{m}$. When $m=1$, a $m$-ballot path is just a Dyck path where up steps and down steps have been replaced by respectively vertical steps and horizontal steps. They are well known combinatorial objects counted by the $m$-Catalan numbers
\begin{equation}
\frac{1}{mn + 1} \binom{(m+1)n}{n}.
\end{equation}

They can also be interpreted as words on a binary alphabet and the notion of  \emph{primitive path} still holds. Indeed, a primitive path is a $m$-ballot path which does not touch the line $y=\frac{x}{m}$ outside its extremal points. From this, the definition of the rotation on Dyck path given in Section \ref{sec:def} can be naturally extended to $m$-ballot-paths, see Figure \ref{fig:mpath-rot}.

\begin{figure}[ht]
\centering
\input{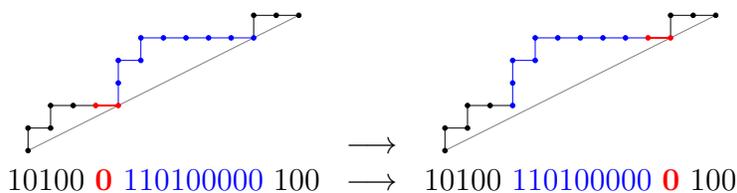}
\caption{Rotation on $m$-ballot paths.}
\label{fig:mpath-rot}
\end{figure}

When interpreted as a cover relation, the rotation on $m$-ballot paths induces a well defined order, and even a lattice \cite{BergmTamari}. This is what we call the $m$-Tamari lattice or $\Tamnm$, see Figure \ref{fig:mTamari} for an example.

\begin{figure}[ht]

  \input{figures/mTamari-3-3-paths}
  
  \caption{$m$-Tamari on $m$-ballot paths: $\Tam{3}{2}$.}
  
  \label{fig:mTamari}
\end{figure}

A formula counting the number of intervals in $\Tamnm$ was conjectured in \cite{BergmTamari} and was proven in \cite{mTamari}. The authors use a functional equation that is a direct generalization of \eqref{eq:functional-equation}. Let $\Phim(x,y)$ be the generating function of intervals of the $m$-Tamari lattice where $y$ is the size $n$ and $x$ a statistic called number of contacts, then \cite[formula (3)]{mTamari} reads\footnote{for consistency with the first part of the article, the $x$ parameter counts the number of contacts minus 1 and so the formula of \cite{mTamari} has been divided by~$x$.}
\begin{equation}
\label{eq:mfunctional-equation}
\Phim(x,y) = 1 + \Bm(\Phi,\Phi, \dots, \Phi),
\end{equation}
where $\Bm$ is a $(m+1)$-linear form defined by 
\begin{align}
\label{eq:m1linear-operator}
\Bm(f, g_1 \dots, g_{m}) &:= xy f \Delta(g_1 \Delta(\dots \Delta(g_{m}))\dots), \\
\Delta(g) &:= \frac{x g(x,y) - g(1,y)}{x - 1}.
\end{align}

Expanding \eqref{eq:functional-equation}, we obtain a sum of $(m+1)$-ary trees. This leads to conjecture that the formula of Theorem \ref{thm:smaller-trees} for counting smaller elements in the lattice generalizes in the $m$-Tamari case, this is indeed true and we prove it in this section.

\subsection{Interpretation in terms of trees}
\label{sec:m-trees}

It was proven in \cite{mTamari} that $\Tamnm$ is actually an upper ideal of $\Tam{n \times m}{1}$. Indeed, there is a natural injection from $m$-ballot paths of size $n$ to Dyck paths of size $mn$ by replacing each vertical step of a $m$-ballot path by $m$ adjacent up steps, see Figure~\ref{fig:m-dyck}. The result set is made of all Dyck paths whose numbers of adjacent up steps are divisible by $m$. We call these paths $m$-Dyck paths and they are in clear bijection with $m$-ballot paths. The set of $m$-Dyck paths is stable by the rotation operation. It is the upper ideal generated by the Dyck word $(1^m.0^m)^n$ which is the image of the $m$-ballot path $(1.0^m)^n$, see Figure \ref{fig:minimal-m-tam}.

\begin{figure}[ht]
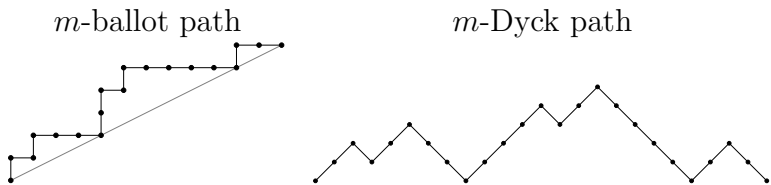

\centering

\def \fpath{figures/}

\begin{tabular}{cc}
$m$-ballot path & $m$-Dyck path \\
\scalebox{0.3}{\input{\fpath mpaths/P6-2-ex1}} &
\scalebox{0.25}{\input{\fpath dyck/D12-ex1}}
\end{tabular}
\caption{A $2$-ballot path and its corresponding $2$-Dyck path.}
\label{fig:m-dyck}
\end{figure}

It is then possible to compute the binary tree image of the minimal $m$-Dyck path by the bijection described in Section \ref{sec:def}. We call this tree the $(n,m)$-comb: it is a left-comb of $n$ right-combs of size $m$, as illustrated in Figure \ref{fig:minimal-m-tam}. 

\begin{figure}[ht]

  \input{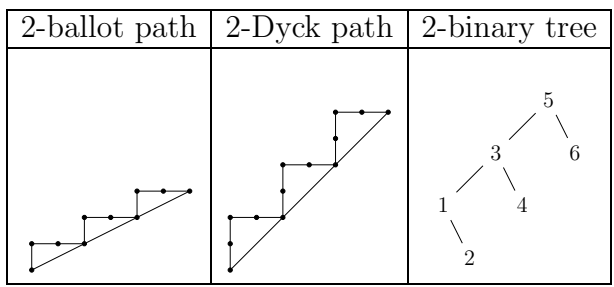}

  \caption{Minimal element of $\mathcal{T}_3^{(2)}$.}

  \label{fig:minimal-m-tam}
\end{figure}

Because the $m$-Tamari lattice corresponds to the upper ideal of the Tamari lattice generated by the word $(1^m.0^m)^n$, it is now clear that it also corresponds to the upper ideal generated by the $(n,m)$-comb, which we give in Figure \ref{fig:mTamari-mbinary}. To fully understand the lattice in terms of trees, two questions remains:
\begin{itemize}
\item How to characterize a binary tree belonging to the upper ideal of the $(n,m)$-comb, \emph{i.e.}, the images of the $m$-Dyck paths?

\item What is the bijection between those trees and the $m$-ballot paths?
\end{itemize}

We can answer both questions by defining a new class of binary trees which we call the \emph{$m$-binary trees}. This definition is crucial to our work on the $m$-Tamari lattice, especially to generalize our results using interval-posets. In this Section, we first give the definition of $m$-binary trees, then prove that they correspond to the upper ideal generated by the $(n,m)$-comb. Finally, we give the explicit bijection between $m$-ballot paths, $m$-binary trees and $(m+1)$-ary trees using the $(m+1)$-ary structure of $m$-binary trees.

\begin{Definition}
\label{def:mbinary}
We define $m$-binary trees recursively by being either the empty binary tree or a binary tree $T$ of size $m \times n$ constructed from $m+1$ subtrees $T_L, T_{R_1}, \dots, T_{R_m}$ such that
\begin{itemize}
\item the sum of the sizes of $T_L, T_{R_1}, \dots, T_{R_m}$ is $m \times (k-1)$,
\item each subtree $T_L, T_{R_1}, \dots, T_{R_m}$ is itself a $m$-binary tree.
\end{itemize}
And $T$ follows the structure bellow,

\begin{center}
\scalebox{0.8}{
\input{figures/m-binary-trees}
}
\end{center}

which we now describe. The left subtree of $T$ is $T_L$. To construct the right subtree $T_R$ of $T$ from $T_{R_1}, \dots, T_{R_m}$, we follow this algorithm:

\begin{algorithmic}
\State $T_R \gets T_{R_1}$
\For{$i = 2$ to $m$}
\State graft a single node $x$ on left of the leftmost node of $T_{R_{i-1}}$
\State graft $T_{R_i}$ on the right of $x$
\EndFor
\end{algorithmic}

\end{Definition}

\begin{figure}[ht]
\centering
\begin{tabular}{c|c}
\scalebox{0.5}{\input{figures/mbinary-example}}
&
\scalebox{0.5}{\input{figures/mbinary-example-2}}
\end{tabular}
\caption{Examples of $m$-binary trees for $m=2$: $T_L$ is in red, $T_{R_1}$ is in dotted blue and $T_{R_2}$ is in dashed green. In the second example, $T_{R_1}$ is empty.}
\label{fig:mbinary}
\end{figure}

Note that $T$ contains the nodes of $T_L, T_{R_1}, \dots, T_{R_m}$ plus exactly $m$ extra nodes. Those nodes are called the \emph{root nodes} of the $m$-binary tree $T$. The actual root of the binary tree $T$ is called the \emph{main root} and the other ones, the \emph{secondary roots}. If $T_{R_i}$ is empty, then the $i^{th}$ root node is directly on the right of the $(i-1)^{th}$ root node. Figure~\ref{fig:mbinary} shows two examples of $m$-binary trees for $m=2$ and the second one illustrates the case where $T_{R_i}$ is empty. The whole set of $m$-binary trees for $m=2$ and $n=3$ is given in Figure \ref{fig:mTamari-mbinary}. 

\begin{figure}[ht]
\centering
\scalebox{0.8}{
\input{figures/mTamari-3-2-mbinary}
}
\caption{The lattice $\Tam{3}{2}$ on $2$-binary tree.}
\label{fig:mTamari-mbinary}
\end{figure}

From the definition, a $m$-binary tree is a special kind of binary tree, just as a $m$-Dyck path is a special kind of Dyck path. The following proposition gives a useful criteria to decide if a given binary tree is indeed a $m$-binary tree. We then use this criteria to prove that $m$-binary trees are the elements of the upper ideal generated by the $(n,m)$-comb. 

\begin{Proposition}
\label{prop:mbinary-carac}
A binary tree $T$ is a $m$-binary tree if and only if it is of size $n \times m$ and its binary search tree satisfies:
\begin{align}
\label{eq:mbinary-cond}
i\cdot m \trprec i\cdot m-1 \trprec \dots \trprec i\cdot m-(m-1)
\end{align}
for all $1 \leq i \leq n$.
\end{Proposition}

This property can be checked in Figure \ref{fig:mbinary-example} which is the binary search tree of the $m$-binary tree of Figure \ref{fig:mbinary}.
\begin{figure}[ht]
\centering
\scalebox{0.6}{
\input{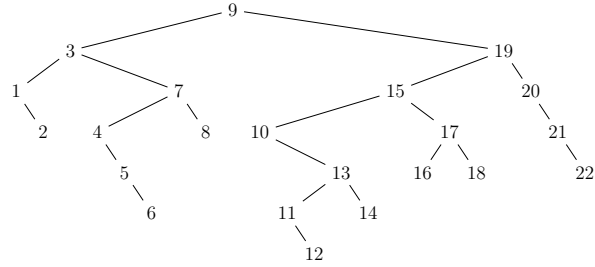}
}
\caption{Binary search tree of a $2$-binary tree. We always have $2k \trprec 2k-1$.}
\label{fig:mbinary-example}
\end{figure}

\begin{proof}
The property is proved by induction on $n$. Let $T$ be a $m$-binary tree composed of  the $m$-binary trees $T_L, T_{R_1}, \dots, T_{R_m}$ which satisfy \eqref{eq:mbinary-cond} by induction hypothesis. We prove that $T$ also satisfies \eqref{eq:mbinary-cond}. The main root of $T$ is labelled by $x = |T_L| + 1$. Because $T_L$ is a m-binary tree, we have that $|T_L| = km$ for some $k \in \NN$, and so $x = km +1$. The $m$-binary tree structure makes it clear that $x+m-1 \trprec x+m-2 \trprec \dots \trprec x$. Besides, the labelling of $T_L$ in $T$ has not been changed, the one of $T_{R_m}$ has been shifted by $|T_L| +m$, the one of $T_{R_{m-1}}$ has been shifted by $|T_L| + |T_{R_m}| + m$ and so on. The labels are only shifted by multiples of $m$ which means that \eqref{eq:mbinary-cond} still holds on $T$.

Now let $T$ be a binary search tree which satisfies \eqref{eq:mbinary-cond}, we have to prove that $T$ satisfies the recursive structure of $m$-binary trees. Let $x$ be the root of $T$. The node $x$ does not precede any element of $T$ so it has to be of the form $x = km + 1$ for some $0 \leq k < n$. Let $T_L$ be the left subtree of $T$, then $|T_L| = km$ and by induction, $T_L$ is a $m$-binary tree. We have $x + 1 \trprec x$, \emph{i.e.}, $x+1$ is in the right subtree of $x$. More precisely, it is the leftmost node of the right subtree. Let $T_{R_1}$ be the binary tree in-between $x$ and $x+1$. For $ 0 \leq a < n$ and $1 \leq b \leq m$ we have that $y = am + b$ is in $T_{R_1}$ if and only if all nodes $am + m \trprec_T am + m-1 \trprec_T \dots \trprec_T am + 1$ are also in $T_{R_1}$. It means $T_{R_1}$ satisfies \eqref{eq:mbinary-cond} and is a $m$-binary tree by induction. The same holds for $T_{R_2}, T_{R_3}, \dots, T_{R_m}$ which gives $T$ the recursive structure of a $m$-binary tree. 
\end{proof}

\begin{Proposition}
\label{prop:m-binary-ideal}
The upper ideal generated by the $(n,m)$-comb is the set of all $m$-binary trees.
\end{Proposition}

\begin{proof}
The final forest of the $(n,m)$-comb is exactly the poset given by \eqref{eq:mbinary-cond}. We proved by Proposition \ref{prop:mbinary-carac} that $m$-binary trees are the binary trees whose final forests are poset extensions of \eqref{eq:mbinary-cond}. This proves the result by using the properties of interval-posets (Proposition~\ref{prop:comb-prop-trivial}).
\end{proof}

By this last proposition, the sub-lattice of the Tamari lattice of size $m \times n$ on $m$-binary trees is indeed the $m$-Tamari lattice. To obtain a $m$-ballot path $W$ from a $m$-binary tree $T$, we first get the Dyck path $D$ corresponding to $T$. This algorithm is just the well known bijection between binary trees and Dyck paths that we gave in Figure~\ref{fig:dyck-tree}. Proposition \ref{prop:m-binary-ideal} assures us that $D$ is actually a $m$-Dyck path, \emph{i.e.}, the numbers of adjacent up steps are always divisible by $m$. This last property can also be proved directly from the bijection and the structure of $m$-binary trees. Because $D$ is a $m$-Dyck path, one can obtain the $m$-ballot path $W$ by replacing each sequence of two consecutive up steps by one vertical step and each down step by an horizontal step. This way, one obtains a $m$-ballot path from a $m$-binary tree. An example of the bijection is given in Figure \ref{fig:path-tree}.

There is also another, equivalent, way to go from a $m$-binary tree to a $m$-ballot path using both $(m+1)$-ary structures of the objects. Indeed, $m$-binary trees are in clear bijection with $(m+1)$-ary trees: the $m$-binary tree $T$, composed from $T_L, T_{R_1}, \dots, T_{R_m}$, is associated to the $(m+1)$-ary tree $T'$ with one root and $m+1$ subtrees $T_L', T_{R_1}', \dots, T_{R_m}'$ respective images of  $T_L, T_{R_1}, \dots, T_{R_m}$. The $m$-ballot paths can also be described by a recursive $(m+1)$-ary structure. A $m$-ballot path $W$ is either the empty path or given by the word
\begin{equation*}
W = W_L~1~W_{R_m}~0~W_{R_{m-1}}~\dots 0~W_{R_1}~0
\end{equation*}
where $W_1, W_{R_1}, W_{R_2}, \dots, W_{R_m}$ are themselves $m$-ballot paths. We obtain the bijection with $m$-binary trees (and $(m+1)$-ary trees) by setting that the empty path is the image of the empty tree and a non-empty $m$-ballot path $W$ is the image of $T$ if and only if $W_1, W_{R_1}, W_{R_2}, \dots, W_{R_m}$ are the images of respectively $T_L, T_{R_1}, \dots, T_{R_m}$. Note that the order of $W_{R_1}, \dots, W_{R_m}$ has to be reversed in $W$ to match the subtrees of $T$ and the first description of the bijection. This is illustrated in Figure \ref{fig:path-tree}.

\begin{figure}[ht]
\centering
\scalebox{0.8}{
\input{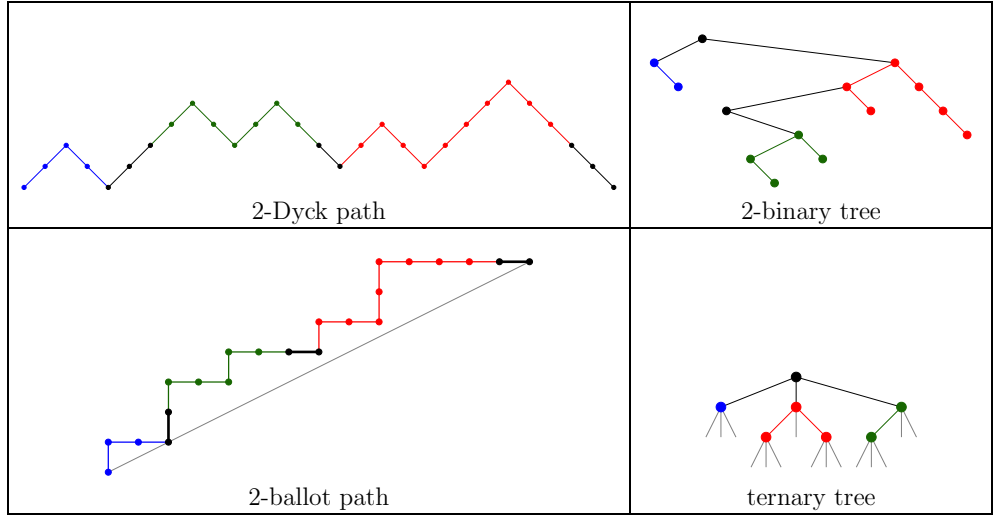}
}
\caption{Bijection between ternary trees and 2-ballot paths: the $2$-binary tree is sent to a Dyck path by the usual bijection, the Dyck path can then be interpreted as a $2$-ballot path. Both the $2$-binary tree and $2$-ballot paths have a ternary structure and are bijectively sent to a ternary tree.}
\label{fig:path-tree}
\end{figure}

The bijection between $m$-binary trees and $(m+1)$-ary trees also allows us to answer a question asked in \cite{mTamari}: what is the description of the $m$-Tamari lattice in terms of $(m+1)$-ary trees? The cover relation on $m$-binary trees is the usual cover relation of the Tamari lattice on binary trees: the right rotation. By Proposition \ref{prop:m-binary-ideal}, we know that it preserves the $m$-binary structure. We can then observe what becomes of $T_L, T_{R_1}, \dots, T_{R_m}$ when applying a rotation on one of the root nodes of $T$. It appears that two different cases arise depending on whether the rotation is made on the main root or on one of the secondary roots. These two cases and their translations in terms of $(m+1)$-ary trees are drawn in Figure \ref{fig:mTamRotation} (for $m=2$). We also give the $m$-Tamari lattices of sizes 2 and 3 for $m=2$ in terms of ternary trees in Figure \ref{fig:mtam-mary}. Note that the description of the lattice in terms of $(m+1)$-ary trees is not needed for the rest of the paper, we will be using only the $m$-binary structure.

\begin{figure}[p]
\begin{fullpage}
  \scalebox{0.8}{\input{figures/mTamari-2-2}}
\input{figures/mTamari-3-2-ternary}

  \caption{$\Tam{2}{2}$ and $\Tam{3}{2}$ on ternary trees.}
  
  \label{fig:mtam-mary}
\end{fullpage}
\end{figure}

\begin{figure}[p]
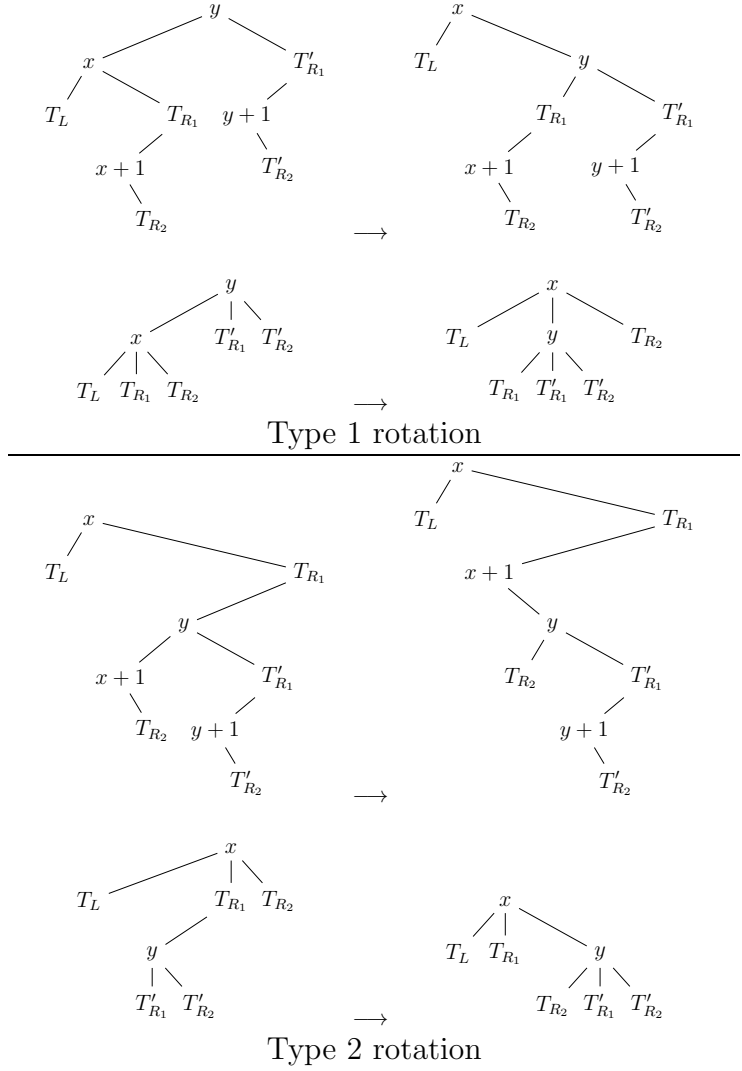

\begin{fullpage}
\begin{tabular}{cc}
\scalebox{0.7}{
\input{figures/mtamari-rotation-type1}
} \\
Type 1 rotation \\
\hline
\scalebox{0.7}{
\input{figures/mtamari-rotation-type2}
} \\
Type 2 rotation
\end{tabular}

\caption{Rotations of type 1 and 2 in $m$-binary trees and $(m+1)$-ary trees.}
\label{fig:mTamRotation}
\end{fullpage}
\end{figure}


\subsection{$m$-Composition and enumeration of intervals}
\label{sec:m-comp}

\begin{Theoreme}
\label{thm:mtamari-functional-equation}
Let $\Phim(x,y)$ be the generating function of intervals of $m$-Tamari where $y$ counts the size of the objects and $x$ the number of touching points of the lowest path (number of contacts with $y=\frac{x}{m}$ after the starting point). Then
\begin{equation}
\Phim = \Bm(\Phim, \dots, \Phim) + 1
\end{equation}
where $\Bm$ is the $(m+1)$-linear operator defined by
\begin{equation}
\label{eq:def-Bm2}
\Bm(f,g_1, \dots, g_m) = f \polleft xy \polright (g_1 \polright ( \dots \polright (g_{m-1} \polright g_m) ) \dots)
\end{equation}
where $\polleft$ and $\polright$ are the left and right products defined in \eqref{eq:polleft-right}.
\end{Theoreme}

This new definition of $\Bm$ \eqref{eq:def-Bm2} is equivalent to the previous one~\eqref{eq:m1linear-operator} and this theorem is just a reformulation of Proposition 8 of \cite{mTamari}. In this section, we propose a new proof by generalizing the concept of interval-poset.

\begin{Definition}
\label{def:m-interval}
A $m$-interval-poset is an interval-poset of size $n\times m$ with
\begin{align}
\label{eq:m-condition}
i\cdot m \trprec i\cdot m-1 \trprec \dots \trprec i\cdot m-(m-1)
\end{align}
for all $1 \leq i \leq n$.
\end{Definition}

\begin{Theoreme}
\label{prop:m-interval}
The $m$-interval-posets of size $n$ are in bijection with intervals of $\Tam{n}{m}$.
\end{Theoreme}

\begin{proof}
A $m$-interval-poset $I$ corresponds to an interval $[T_1, T_2]$ of the Tamari order of size $n \times m$. By Proposition \ref{prop:mbinary-carac}, the binary tree $T_1$ is a $m$-binary tree. As $T_2 \geq T_1$, then $T_2$ is also a $m$-binary tree and $I$ is an interval of $\Tamnm$. 
\end{proof}

The number of nodes on the border of a $m$-binary tree is the same as on its associated $(m+1)$-ary tree and still corresponds to the number of touch points of the $m$-ballot path. We then define
\begin{equation}
\PIm(I) := x^{\Itrees(I)}y^{\frac{\Isize(I)}{m}},
\end{equation}
and we have
\begin{equation}
\Phim(x,y) = \sum_{I} \PIm(I)
\end{equation}
summed on all $m$-interval-posets.

By composing two $m$-interval-posets, one does not obtain a sum on $m$-interval-posets: the sizes are not multiples of $m$ any more. We have to generalize the $\BB$ composition to a $m$-composition which has to be a $(m+1)$-linear operator. A simple translation of \eqref{eq:def-Bm2} in terms of $\pleft$ and $\pright$ is not enough. Indeed, it wouldn't generate all $m$-interval-posets. However, the following expression
\begin{equation}
g_1 \polright g_2 = \frac{(xy \polright g_2) \polleft g_1}{xy}.
\end{equation}
reflects the $m$-binary structure given by Figure \ref{fig:mbinary}. We then rewrite \eqref{eq:def-Bm2} by using this observation. As an example, for $m=3$, one obtains
\begin{align*}
\Bk{3}(f,g_1,g_2,g_3) &= f \polleft xy \polright \left( g_1 \polright \left( g_2 \polright g_3 \right) \right) \\
&= f \polleft xy \polright \frac{1}{xy} \left( \left( xy \polright \frac{1}{xy}\left( (xy \polright g_3) \polleft g_2 \right) \right) \polleft g_1 \right) \\
&= \frac{1}{y^2} \left( f \polleft xy \polrightx \left( \left( xy \polrightx \left( (xy \polright g_3) \polleft g_2 \right) \right) \polleft g_1 \right) \right)
\end{align*}
where
\begin{align}
f \polrightx g &:= f \polright (\frac{g}{x}) 
= f \Delta(\frac{g}{x}).
\end{align}
The $\polrightx$ operation can be interpreted on interval-posets.

\begin{Definition}
\label{def:polrightx}
Let $I_1$ and $I_2$ be two interval-posets such that $\Itrees(I_2) =k$. Let $y$ be the maximal label of $I_1$ and $x_1, \dots, x_k$ be the roots of $\dec(I_2)$. Then $I_1 \prightx I_2$ is the sum of the $k$ interval-posets $P_1, \dots, P_k$  where $P_i$ is the shifted concatenation of $I_1$ and $I_2$ with the $i$ added decreasing relations: $x_j \trprec y$ for $j \leq i$.
\end{Definition}

The sum $I_1 \prightx I_2$ is just the sum $I_1 \pright I_2$ of Definition \ref{def:left-right-product} minus the $P_0$ poset (the shifted concatenation with no extra decreasing relation). In particular, this means that the obtained interval-posets all have the relation $2 \trprec 1$ because 1 is always the minimal root of $\dec(I_2)$.

\begin{Proposition}
\label{prop:mcomposition}
The $(m+1)$-linear operator $\BBm$ on $m$-interval-posets is defined by
\begin{equation*}
\BBm(I_L, I_{R_1}, I_{R_2}, \dots, I_{R_m}) := I_L \pleft~u~\prightx \left( (u~\prightx (u \prightx \dots ((u \pright I_{R_m}) \pleft I_{R_{m-1}}) \pleft \dots )\pleft I_{R_1} \right) 
\end{equation*}
where $u$ the interval-poset containing a single vertex. Recursively, the definition reads
\begin{equation*}
\BBm(I_L, I_{R_1}, \dots, I_{R_m}) := I_L \pleft \BR(I_{R_1}, \dots, I_{R_m})
\end{equation*}
with
\begin{align*}
\BR(I) &:= u \pright I, \\
\BR(I_1, \dots, I_k) &:=  u \prightx \left( \BR(I_2, \dots, I_k) \pleft I_1 \right).
\end{align*}
The result is a sum of $m$-interval-posets. The $\BBm$ operator is
called the $m$-composition.
\end{Proposition}

Note that we give the Sage code corresponding to this definition in Appendix \ref{app:sub-sec:mcomp}.

\begin{proof}
Let us first notice that we compose with the interval-poset $u$ exactly $m$ times. It means that $m$ vertices have been added to $I_L, I_{R_1}, I_{R_2},\allowbreak \dots, I_{R_m}$: the size of the obtained intervals are multiples of $m$.

The first operation is $\BR(I_{R_m}) = u \pright I_{R_m}$ which is a sum of interval-posets of size $1 + |I_{R_m}|$. The labels of $I_{R_m}$ have been shifted by 1. The next operation is
\begin{equation}
\BR(I_{R_{m-1}}, I_{R_m}) = u \prightx \left( \BR(I_{R_m}) \pleft I_{R_{m-1}} \right).
\end{equation}

The computation $\BR(I_{R_m}) \pleft I_{R_{m-1}}$ consists of attaching $I_{R_{m-1}}$ to the interval-posets of $u \pright I_{R_m}$ without adding any decreasing relations. The labels of $I_{R_{m-1}}$ are shifted by $1 + |I_{R_m}|$. By doing $ u \prightx \left( \BR(I_{R_m}) \pleft I_{R_{m-1}} \right)$, we obtain a sum of interval-posets which all have the relation $2 \trprec 1$. The labels of $I_{R_m}$ have been shifted by 2 and the those of $I_{R_{m-1}}$ by $2 + |I_{R_m}|$.

By redoing this operation, we obtain that $\BR(I_{R_1}, \dots, I_{R_m})$ is a sum of interval-posets which all have the relations $m \trprec m-1 \trprec \dots \trprec 1$. The labels of $I_{R_m}$ are shifted by $m$, those of $I_{R_{m-1}}$ by $m + |I_{R_m}|$ and so on until $I_{R_1}$ whose labels have been shifted by  $m + |I_{R_2}| + \dots + |I_{R_m}|$. This means that  $\BR(I_{R_1}, \dots, I_{R_m})$ is an $m$-interval-poset. So is $I_L \pleft \BR(I_{R_1}, \dots, I_{R_m})$ because the left product on two $m$-interval-posets is still a $m$-interval-poset. 
\end{proof}

As an example, here is a detailed computation for $m=2$.

\input{figures/mcomposition}

\begin{Proposition}
\label{prop:desc-mcomp}
Let $I_L, I_{R_1}, \dots, I_{R_m}$ be some $m$-interval-posets. The $m$-interval-poset $I_0$ is defined by
\begin{enumerate}[label=(\roman{*}), ref=(\roman{*})]
\item $I_0$ is a poset extension of the shifted concatenation of $I_L$, $r$, $I_{R_m}, I_{R_{m-1}}, \dots, I_{R_1}$ where $r$ is the poset $m \trprec m-1 \trprec \dots \trprec 1$.
\label{def:mcomp:cond:I1}
\item For $k = |I_L| + 1$, we have $i \trprec k$ for all $i \in I_L$ 
\label{def::mcomp:cond:increasingIL}
\item For all $j$ such that $1 \leq j <m$, if $I_{R_j}$ is not empty then we set $a_j$ to be the minimal label of $I_{R_j}$ and we have $i \trprec a_j$ for all $i$ such that $a_j > i > k+j$.
\label{def:mcomp:cond:increasingIR}
\item $I_0$ does not have any other relations. 
\end{enumerate}
Then  $\BBm(I_L, I_{R_1}, \dots, I_{R_m})$ is the sum of the $m$-interval-posets $I$ of size  $m + |I_L| + |I_{R_1}| + \dots + |I_{R_m}|$ such that $I$ is a poset extension of $I_0$ on which only decreasing relations have been added and no relations have been added inside the subposets $I_L, I_{R_1}, \dots, I_{R_m}$. 
\end{Proposition}

\begin{proof}
The construction of $\BBm(I_L, I_{R_1}, \dots, I_{R_m})$ follows the structure of a $m$-binary tree. Let $T_L, T_{R_1}, \dots, T_{R_m}$ and  $T_L', T_{R_1}', \dots, T_{R_m}'$ be respectively the lower and upper $m$-binary trees of the intervals $I_L,\allowbreak I_{R_1}, \dots, I_{R_m}$. And let $T$ and $T'$ be respectively the lower and upper trees of $I_0$. Then, because of the increasing relations of $I_0$, the $m$-binary tree $T'$ is the one formed by $T_L', T_{R_1}', \dots, T_{R_m}'$ as in Figure~\ref{fig:mbinary}. This is the common upper tree of all the intervals obtained by $\BBm(I_L, I_{R_1}, \dots, I_{R_m})$. Indeed, increasing relations are the same for all intervals: $\pleft$ corresponds to a plug on left and $ \pright$ and $\prightx$ to a plug on the right. The interval $I_0$ is actually the interval of $\BBm(I_L, I_{R_1}, \dots, I_{R_m})$ with the minimal number of decreasing relations. Indeed, in terms of decreasing relations, it corresponds by definition to a concatenation of $I_L$, $r$, $I_{R_m}, I_{R_{m-1}}, \dots, I_{R_1}$ where $r$ is the poset $m \trprec m-1 \trprec \dots \trprec 1$ which is what we obtain from  $\BBm(I_L, I_{R_1}, \dots, I_{R_m})$.

Now, the intervals satisfying Proposition~\ref{prop:desc-mcomp} are all possible ways of adding decreasing relations to $I_0$ toward vertices of $k+m-1 \trprec k+m-2 \trprec \dots \trprec k$. Indeed, because of the increasing relations, there can not be any decreasing relations in-between intervals $I_L,\allowbreak I_{R_1}, \dots, I_{R_m}$. By definition of $\pright$ and $\prightx$ it is then clear that the intervals of $\BBm(I_L, I_{R_1},\allowbreak \dots, I_{R_m})$ are exactly the extensions of $I_0$ as defined by Proposition~\ref{prop:desc-mcomp}.
\end{proof}

As explained in the proof, the intervals resulting of a $m$-compositions all share the same maximal tree given by the structure of the $m$-binary tree. The minimal trees range from a tree where all minimal trees of composed intervals have been grafted at the left of one another to the $m$-binary tree formed by all minimal trees. This illustrated in the case where $m=2$ in Figure~\ref{fig:mcomp-intervals}.

\begin{figure}[ht]
\centering
\scalebox{0.8}{
\input{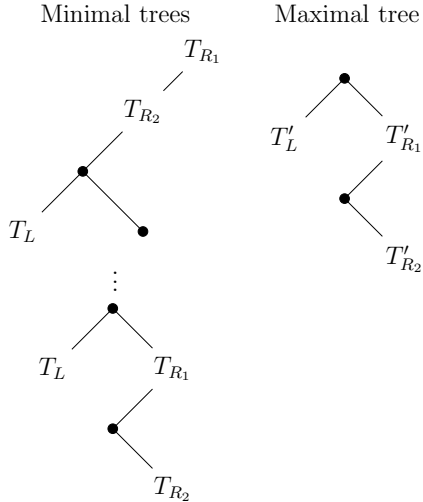}
}
\caption{Minimal and maximal trees of the intervals of a $m$-composition.}
\label{fig:mcomp-intervals}
\end{figure}

\begin{Proposition}
\label{prop:combinatorial-equivalence-mcomp}
Let $I_L, I_{R_1}, \dots, I_{R_m}$ be $m$-interval-posets. Then
\begin{equation*}
\PIm(\BBm(I_L,I_{R_1}, \dots, I_{R_m})) = \Bm(\PIm(I_L),\PIm(I_{R_1}), \dots, \PIm(I_{R_m}))
\end{equation*}
\end{Proposition}

\begin{proof}
The only thing to prove is
\begin{equation}
\label{eq:equiv-prightx}
\PI(I_1 \prightx I_2) = \PI(I_1) \polrightx \PI(I_2)
\end{equation}
for $I_1$ and $I_2$ two interval-posets. Indeed, let $Y= y^{\frac{1}{m}}$ and $I$ be a $m$-interval-poset of size $nm$, then
\begin{equation*}
\PIm(I)(x,y) = \PI(I)(x,Y).
\end{equation*}
And so if \eqref{eq:equiv-prightx} is satisfied, so are \eqref{eq:equiv-pleft} and \eqref{eq:equiv-pright} and we have
\begin{align*}
&\PIm(\BBm(I_L, I_{R_1}, \dots, I_{R_m})) = \PI \left( \BBm(I_L, I_{R_1}, \dots, I_{R_m}) \right)(x,Y) \\
&= \PI \left( I_L \pleft~u~\prightx \left( (u~\prightx \dots ((u \pright I_{R_m}) \pleft I_{R_{m-1}}) \pleft \dots )\pleft I_{R_1} \right) \right)(x,Y) \\
&= \PI(I_L) \polleft~x.Y~\polrightx \left( (x.Y~\polrightx \dots ((x.Y \polright \PI(I_{R_m})) \polleft \PI(I_{R_{m-1}})) \polleft \dots )\polleft \PI(I_{R_1}) \right) \\
&= Y^{m-1} \PI(I_L) \polleft x.Y \polright \left( \PI(I_{R_1}) \polright \dots \polright (\PI(I_{R_{m-1}}) \polright \PI(I_{R_m}))) \dots \right) \\
&= \Bm(\PI(I_L),\PI(I_{R_1}), \dots, \PI(I_{R_m}))(x,Y) \\
&= \Bm(\PIm(I_L),\PIm(I_{R_1}), \dots, \PIm(I_{R_m})).
\end{align*}

We then prove \eqref{eq:equiv-prightx}. We set $k := \Itrees(I_2)$, we have
\begin{align*}
\Delta \left( \frac{\PI(I_2)}{x} \right) &= \Delta(y^{\Isize(I_2)} x^{k-1}) \\
&= y^{\Isize(I_2)} (1 + x + x^2 + \dots + x^{k-1}), \\
\PI(I_1) \polrightx \PI(I_2) &= y^{\Isize(I_1) + \Isize(I_2)} x^{\Itrees(I_1)}(1 + x + x^2 + \dots + x^{k-1})
\end{align*}
Besides, $I_1 \prightx I_2$ is the sum of interval-posets $P_i$, $1 \leq i \leq k$ where  $\Isize(P_i) = \Isize(I_1) + \Isize(I_2)$ and $\Itrees(P_i) = \Itrees(I_1) + k -i$ which proves the result.
\end{proof}

We can check \eqref{eq:equiv-prightx} on \eqref{eq:mtamari:exemple-prightx}.
\begin{align*}
xy \polrightx y^7(x^4 + x^3) &= y^8x(1+x+x^2 +x^3 + 1 + x +x^2) \\
&= y^8(2x + 2x^2 + 2x^3 + x^4)  
\end{align*}

Besides, by computing
\begin{align*}
\Bm(xy,x^2y^2,xy) &= xy \polleft (xy \polright (x^2y^2 \polright xy)) \\
&= y^5 x  (x \polright x^2(1+x)) \\
&= y^5x^2 (1 + x +x^2 + 1 + x + x^2 + x^3) \\
&= y^5(2x^2 + 2x^3 + 2x^4 + x^5),
\end{align*}
we check the result on \eqref{eq:mtamari:exemple-mcomp}.

\begin{Proposition}
\label{prop:unicity-mcomp}
Let $I$ be a $m$-interval-poset, then there is exactly one list $I_1, \dots, I_{m+1}$ of $m$-interval-posets such that $I$ appears in the $m$-composition $\BBm(I_1, \dots, I_{m+1})$.
\end{Proposition}

\begin{proof}
We define $k$ the same way as in the proof of Proposition \ref{prop:unicity-composition}: $k$ is the maximal label such that $i \trprec k$ for all $i < k$. And for the same reasons, $k$ is unique and $I_L$ is made of vertices $i < k$. For $1 \leq j < m$, let $a_j$ be the minimal label such that $k+j+1 \trprec a_j$ and $k + j \ntrprec a_j$. If there is no such label, we set $a_j := \emptyset$. And let $a_m := k+m$ if $k+m-1 \ntrprec k+m$ or $\emptyset$ otherwise.

The vertices $a_1, \dots, a_m$ satisfy Condition~\ref{def:mcomp:cond:increasingIR} of Proposition~\ref{prop:desc-mcomp}. They allow us to cut $I$ into $m+1$ subposets. If $a_j = \emptyset$ then $I_{R_j} = \emptyset$, otherwise $I_{R_j}$ is the subposet of $I$ of which $a_j$ is the minimal label.

All conditions of Proposition~\ref{prop:desc-mcomp} are satisfied and so $I \in \BBm(I_L,\allowbreak I_{R_1}, \dots, I_{R_m})$. Besides, the vertices $a_1, \dots, a_m$
 are the only one to satisfy Condition \ref{def:mcomp:cond:increasingIR} of Proposition \ref{prop:desc-mcomp} without adding any increasing relations to $I_0$: they give the only way to cut the poset $I$.
\end{proof}

\begin{proof}[Proof of Theorem~\ref{thm:mtamari-functional-equation}]
The proof is direct by Propositions~\ref{prop:combinatorial-equivalence-mcomp} and \ref{prop:unicity-mcomp} by the same reasoning as for $m=1$ of Theorem~\ref{thm:functional-equation}.
\end{proof}

With Propositions \ref{prop:combinatorial-equivalence-mcomp} and \ref{prop:unicity-mcomp}, we now have a new proof of the functional equation \eqref{eq:mfunctional-equation} already described in \cite{mTamari}. We can go further and give a generalized version of Theorem \ref{thm:smaller-trees}.

\subsection{Counting smaller elements in $m$-Tamari}
\label{sec:smaller-mtrees}

\begin{Proposition}
\label{prop:sum-mcomposition}
Let $T$ be a $m$-binary tree and $S_T := \sum_{T' \leq T} P_{[T',T]}$, the sum of all $m$-interval-posets with maximal tree $T$. If $T$ is composed of the $m$-binary trees 
 $T_L, T_{R_1}, \dots, T_{R_m}$, then $S_T = \BBm(S_{T_L},\allowbreak S_{T_{R_1}}, \dots, S_{T_{R_m}})$.
\end{Proposition}

\begin{proof}
Let $I_0$ be the interval $[T_0, T]$ where $T_0$ is the $(n,m)$-comb, \emph{i.e.}, the minimal $m$-binary tree. The increasing relations of $I_0$ are the ones of $T$ and the decreasing relations are \eqref{eq:mbinary-cond}. We cut $I_0$ into the subposets $I_L, I_{R_1}, I_{R_2}, \dots, I_{R_m}$ following the cutting of $T$ (the labels of $I_{R_j}$ in $I_0$ are the ones of $T_{R_j}$ in $T$). By construction, the $m$-interval-posets $I_L, I_{R_1}, \dots, I_{R_m}$ are the initial $m$-Tamari intervals of respectively $T_L, T_{R_1}, \dots, T_{R_m}$.

Let $P$ be an interval-poset of the sum $S_T$, \emph{i.e.} an poset extension of $I_0$ where only decreasing relations have been added. By cutting $P$ in the same way than $I_0$, then $P_L, P_{R_1}, \dots, P_{R_m}$ are extensions of respectively  $I_L, I_{R_1}, \dots, I_{R_m}$ and so appear respectively in  $S_{T_L}, S_{T_{R_1}}, \dots, S_{T_{R_m}}$. And by Proposition~\ref{prop:desc-mcomp}, because the increasing relations of $P$ are those of $I_0$, then $P \in \BBm(P_L, P_{R_1}, \dots, P_{R_m})$.

Conversely, if  $P_L, P_{R_1}, \dots, P_{R_m}$ are elements of respectively $S_L, S_{R_1},\allowbreak \dots, S_{R_m}$, then the increasing relations of the elements of 
$\BBm(P_L,\allowbreak P_{R_1}, \dots, P_{R_m})$ are those of $T$ which make them elements of $S_T$.
\end{proof}

\begin{figure}[ht]
\scalebox{0.8}{
  \input{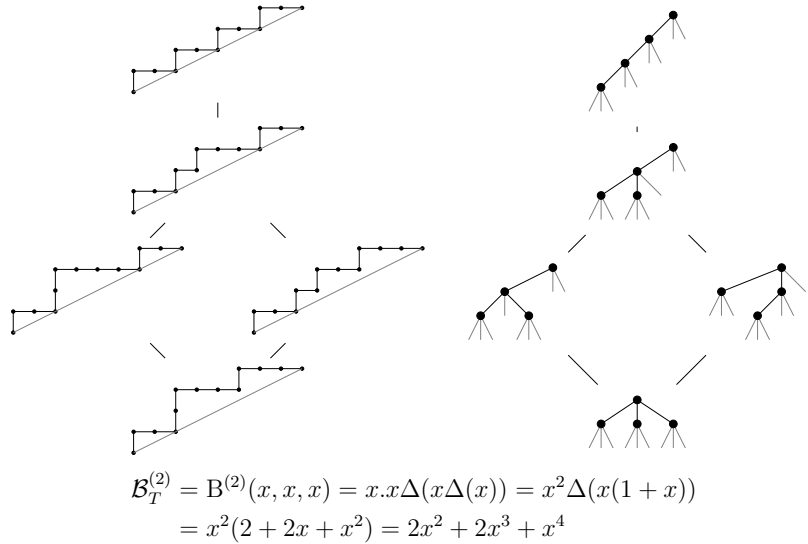}
  }

  \caption{Example of $\BTm{T}$ computation. We compute
    $\BTk{T}{2}$ for the tree at the bottom of the graph and obtain
    $\BTk{T}{2}(1) = 5$ which corresponds to the number of elements
    smaller than or equal to $T$.}

  \label{fig:BmTExemple}

\end{figure}

\begin{Theoreme}
\label{thm:smaller-mtrees}
Let $T$ be a $(m+1)$-ary tree, we define recursively $\BTm{T}(x)$ by:
\begin{align*}
\BTm{\emptyset} &:= 1\text{, and} \\
\BTm{T} &:= \Bm_{y=1}(\BTm{T_L},\BTm{T_{R_1}}, \dots, \BTm{T_{R_m}})
\end{align*}
where $T_L, T_{R_1} \dots, T_{R_m}$ are the subtrees of $T$. Then $\BTm{T}(x)$ counts the number of elements smaller than $T$ in $\Tamnm$ according to the number of nodes on their leftmost branch (or the number of contacts on their ballot-path). In particular, $\BTm{T}(1)$ is the number of elements smaller than $T$ in $\Tamnm$.
\end{Theoreme}

See an example of this computation of Figure \ref{fig:BmTExemple}: one can check that the power of $x$ corresponds either to
the number of nodes on the leftmost branch of the tree or to the number of
contacts minus 1 on ballot paths.

\begin{proof}
As in Theorem~\ref{thm:smaller-trees}, we want to prove
\begin{equation*}
\BTm{T} = \PIm(S_T).
\end{equation*}
The result is obtained by an induction on $n$ by Propositions~\ref{prop:combinatorial-equivalence-mcomp} and \ref{prop:sum-mcomposition}
\begin{align*}
\BTm{T} &= \Bm(\BTm{T_L}, \BTm{T_{R_1}}, \dots, \BTm{T_{R_m}}) \\
&= \Bm(\PIm(S_{T_L}), \PIm(S_{T_{R_1}}), \dots, \PIm(S_{T_{R_m}})) \\
&= \PIm\left( \BBm( S_{T_L}, S_{T_{R_1}}, \dots, S_{T_{R_m}} ) \right) \\
&= \PIm(S_T).
\qedhere
\end{align*}
\end{proof}

\subsubsection*{Acknowledgements}
This work has been partially funded by the \emph{SFB F50, Algorithmic and Enumerative Combinatorics}.
The computation and tests needed along the research were done using the open-source mathematical software \texttt{Sage}~\cite{SAGE_WEBSITE} and its combinatorics features developed by the \texttt{Sage-Combinat} community \cite{SAGE_COMBINAT}. 

\bibliographystyle{plain}
\label{sec:biblio}
\bibliography{long-version}

\begin{appendix}
\section{Sage implementation of interval-posets}

This code is also available from the author's webpage as a Demo sage worksheet on SageMathCloud \cite{SAGE_Demo}.

\label{app:sage-interval-posets}
\subsection{Basic example}
\label{app:sub-sec:basic}

Below is the sage code to create an interval-poset, compute its endpoints as binary trees and the list of Dyck paths in the interval. 
\begin{lstlisting}
sage: ip = TamariIntervalPoset(4,[(2,1),(3,1),(2,4),(3,4)]); ip
The tamari interval of size 4 induced by relations [(2, 4), (3, 4), (3, 1), (2, 1)]
sage: view(ip)
sage: ip.lower_binary_tree()
[[., [[., .], .]], .]        
sage: ip.upper_binary_tree()
[., [[., [., .]], .]]
sage: list(ip.dyck_words())
[[1, 1, 1, 0, 0, 1, 0, 0],
 [1, 1, 1, 0, 0, 0, 1, 0],
 [1, 1, 0, 1, 0, 1, 0, 0],
 [1, 1, 0, 1, 0, 0, 1, 0]]
sage: IP4 = TamariIntervalPosets(4); IP4
Interval-posets of size 4
sage: IP4.cardinality()
68
sage: ip in IP4
True
\end{lstlisting}

\subsection{Composition}
\label{app:sub-sec:comp}
The composition function is not yet included in the {\tt TamariIntervalPoset} package. But it can be easily coded using left and right products.

\begin{lstlisting}[firstline=2,lastline=25]
def left_product(ip1,ip2):
    size = ip1.size() + ip2.size()
    # Juxtaposition of ip1 and shifted ip2
    relations = list(ip1._cover_relations) + [(i+ip1.size(),j+ip1.size()) for (i,j) in ip2._cover_relations]
    # Increasing relations between ip1 and the first vertex of ip2
    relations+= [(i,ip1.size()+1) for i in ip1.increasing_roots()]
    return TamariIntervalPoset(size,relations)
    
def right_product(ip1, ip2):
    size = ip1.size() + ip2.size()
    # Juxtaposition of ip1 and shifted ip2
    relations = list(ip1._cover_relations) + [(i+ip1.size(),j+ip1.size()) for (i,j) in ip2._cover_relations]
    # First element: no extra decreasing relation
    yield TamariIntervalPoset(size,relations)
    for j in ip2.decreasing_roots():
        # Adding decreasing relations 1 by 1
        relations.append((j+ip1.size(),ip1.size()))
        yield TamariIntervalPoset(size,relations)
        
def composition(ip1,ip2):
    u = TamariIntervalPoset(1,[])
    left = left_product(ip1,u)
    for r in right_product(left,ip2):
        yield r
\end{lstlisting}

Here is how we now obtain the computation of Figure~\ref{fig:composition}.

\begin{lstlisting}
sage: ip1 = TamariIntervalPoset(3,[(1,2),(3,2)])
sage: ip2 = TamariIntervalPoset(4,[(2,3),(4,3)])
sage: list(composition(ip1,ip2))
[The tamari interval of size 8 induced by relations [(1, 2), (2, 4), (3, 4), (6, 7), (8, 7), (3, 2)],
 The tamari interval of size 8 induced by relations [(1, 2), (2, 4), (3, 4), (6, 7), (8, 7), (5, 4), (3, 2)],
 The tamari interval of size 8 induced by relations [(1, 2), (2, 4), (3, 4), (6, 7), (8, 7), (6, 4), (5, 4), (3, 2)],
 The tamari interval of size 8 induced by relations [(1, 2), (2, 4), (3, 4), (6, 7), (8, 7), (7, 4), (6, 4), (5, 4), (3, 2)]]
\end{lstlisting}

\end{appendix}

\subsection{$m$-Composition}
\label{app:sub-sec:mcomp}

To obtain the $m$-composition, we add a function corresponding to $\prightx$ and then follow the definition of  Proposition~\ref{prop:mcomposition}.

\begin{lstlisting}[firstline=27,lastline=49]
def right_product_dim(ip1,ip2):
    size = ip1.size() + ip2.size()
    # Juxtaposition of ip1 and shifted ip2
    relations = list(ip1._cover_relations) + [(i+ip1.size(),j+ip1.size()) for (i,j) in ip2._cover_relations]
    for j in ip2.decreasing_roots():
        # Adding decreasing relations 1 by 1
        relations.append((j+ip1.size(),ip1.size()))
        yield TamariIntervalPoset(size,relations)
        
def mcomposition(ips):
    rights = ips[1:] # we take the list of intervals excpect the first one
    def compute_rights(rights): # we define a recursive method to compute the right part
        u = TamariIntervalPoset(1,[])
        if len(rights)==1:
            for r in right_product(u,rights[0]):
                yield r
        else:
            for r1 in compute_rights(rights[1:]):
                r1 = left_product(r1,rights[0])
                for r2 in right_product_dim(u,r1):
                    yield r2
    for r in compute_rights(rights):
        yield left_product(ips[0],r)
\end{lstlisting}

We can now compute the example of \eqref{eq:mtamari:exemple-mcomp}.

\begin{lstlisting}
sage: ip1 = TamariIntervalPoset(2,[(2,1)])
sage: ip2 = TamariIntervalPoset(4,[(2,1),(4,3),(2,3)])
sage: ip3 = ip1
sage: list(mcomposition([ip1,ip2,ip3]))
[The tamari interval of size 10 induced by relations [(1, 3), (2, 3), (4, 7), (5, 7), (6, 7), (8, 9), (10, 9), (8, 7), (6, 5), (4, 3), (2, 1)],
 The tamari interval of size 10 induced by relations [(1, 3), (2, 3), (4, 7), (5, 7), (6, 7), (8, 9), (10, 9), (8, 7), (6, 5), (5, 3), (4, 3), (2, 1)],
 The tamari interval of size 10 induced by relations [(1, 3), (2, 3), (4, 7), (5, 7), (6, 7), (8, 9), (10, 9), (8, 7), (7, 3), (6, 5), (5, 3), (4, 3), (2, 1)],
 The tamari interval of size 10 induced by relations [(1, 3), (2, 3), (4, 7), (5, 7), (6, 7), (8, 9), (10, 9), (9, 3), (8, 7), (7, 3), (6, 5), (5, 3), (4, 3), (2, 1)],
 The tamari interval of size 10 induced by relations [(1, 3), (2, 3), (4, 7), (5, 7), (6, 7), (8, 9), (10, 9), (8, 7), (6, 5), (5, 4), (4, 3), (2, 1)],
 The tamari interval of size 10 induced by relations [(1, 3), (2, 3), (4, 7), (5, 7), (6, 7), (8, 9), (10, 9), (8, 7), (7, 3), (6, 5), (5, 4), (4, 3), (2, 1)],
 The tamari interval of size 10 induced by relations [(1, 3), (2, 3), (4, 7), (5, 7), (6, 7), (8, 9), (10, 9), (9, 3), (8, 7), (7, 3), (6, 5), (5, 4), (4, 3), (2, 1)]]
\end{lstlisting}

\end{document}